%% file: main.tex
\documentclass[review,onefignum,onetabnum]{siamart251216}

\input{ex_shared}
\usepackage{xr}
\nolinenumbers

\begin{document}
\maketitle

\begin{abstract}
 In this work, we develop an accurate numerical homogenization method for
computing effective Hamiltonians of quasiperiodic Hamilton--Jacobi equations
(QHJEs) with convex Hamiltonians of the form
$H(x,p) = |p|^k/{k}-f(x),  ~k>1$,
where $f$ is quasiperiodic. Computing effective Hamiltonians in the quasiperiodic setting requires solving QHJEs posed on the whole space. Their solutions generally possess neither translational symmetry nor decay and may exhibit low regularity. These features pose substantial challenges for numerical
computation. To address these difficulties, we introduce a quasiperiodic boundary
condition, which allows the original whole-space problem to be treated on a
bounded domain while preserving quasiperiodicity at the boundary. We then propose an SL--FPR scheme that combines a semi-Lagrangian (SL) approximation with
the finite points recovery (FPR) method and establish stability and error estimates for the resulting scheme.  We also extend the quasiperiodic homogenization
result from the quadratic case to general $k>1$ and apply the proposed method to accurately approximate
the corresponding effective Hamiltonians. Numerical experiments
illustrate the convergence and applicability of the method and validate the extended homogenization results.

\end{abstract}

\begin{keywords}
	Quasiperiodic Hamilton--Jacobi equations, Quasiperiodic boundary conditions, Finite points recovery  method,   Quasiperiodic homogenization, Effective Hamiltonian.
\end{keywords}
	
\begin{AMS}
	35F21, 65M12, 65M25, 35B27
\end{AMS}
	
\section{Introduction}\label{sec:introduction}

Quasiperiodic Hamilton-Jacobi equations (QHJEs) arise in a wide range of scientific and engineering problems, including celestial and classical mechanics, dynamical systems, as well as models of non-Newtonian and complex fluids \cite{levine1984quasicrystals,poincare1890problem,sochi2014using,kamrin2014symmetry}. Beyond these physical applications, QHJEs are also deeply connected to the theoretical foundations of optimal control theory, forming an essential component in the analysis of value functions \cite{bardi1997optimal,deville1993smooth}. Specifically, QHJEs take the following form 
\begin{equation}
\begin{cases}
u_t+H(x,u_x)=0,
& \text{in } \mathbb{R}\times(0,\infty),\\
u(x,0)=u_0(x),
& \text{on } \mathbb{R}\times\{t=0\},
\end{cases}
\label{eqn:QHJE}
\end{equation}
where \(u_0\in\mathbf{BUC}(\mathbb{R})\cap\mathbf{Lip}(\mathbb{R})\) is quasiperiodic (see \Cref{def:QP_fun}). For the Hamiltonian, we focus on the \(k\)-power convex form
\begin{equation}
H(x,p)=\frac{|p|^k}{k}-f(x),
\qquad k>1,
\label{eqn:k_hamiltonian}
\end{equation}
which satisfies the following assumptions
\begin{enumerate}[label={(\alph*)}]
\item \(f\) is quasiperiodic, and hence
\(H(\cdot,p)\) is quasiperiodic for each fixed \(p\in\mathbb R\);
\label{itm:quasiperiodic}
\item \(H\) is uniformly coercive in \(p\), that is,
\(    \lim\limits_{|p|\to\infty}\displaystyle\frac{H(x,p)}{|p|}=\infty;
   \)
\label{itm:coercive}
\item \(H\) is convex in \(p\): for each fixed \(x\in\mathbb R\), the map \(p\mapsto H(x,p)\) is convex.
\label{itm:convex}
\end{enumerate}

In particular, the QHJE \eqref{eqn:QHJE} also plays an important role in
quasiperiodic homogenization. As the microscopic scale
\(\varepsilon\to0\), the solution of the multiscale Hamilton--Jacobi
equation converges to that of an effective equation governed by an
effective Hamiltonian \(\overline H\) \cite{ishii2000almost,tran2021hamilton}. For a fixed \(p\in\mathbb{R}\), \(\overline H(p)\) is the unique constant such that the following \(\delta\)-approximate corrector problem admits a viscosity solution \(v\)
\begin{equation}
\overline H(p)-\delta
\leq
H\bigl(x,p+v'(x)\bigr)
\leq
\overline H(p)+\delta,
\qquad x\in\mathbb{R},
\label{eqn:approx_corrector}
\end{equation}
for any \(\delta>0\).  In general, the corresponding exact corrector problem obtained by setting \(\delta=0\) may not admit a solution \cite{lions2003correctors}. In the quadratic case \(k=2\) of \eqref{eqn:k_hamiltonian}, however, an exact corrector exists and satisfies \cite{hu2024polynomial}
\begin{equation}
H\bigl(x,p+v'(x)\bigr)=\overline H(p),
\qquad x\in\mathbb{R}.
\label{eqn:exact_corrector}
\end{equation}
Besides the above corrector formulation, another approach to determining the effective Hamiltonian is based on the long-time behavior of the corresponding time-dependent Hamilton--Jacobi equation \cite{giga2021existence}. More precisely, consider the Cauchy problem 
\begin{equation}
\begin{cases}
w_t+H(x,p+w_x)=0,
& (x,t)\in\mathbb{R}\times(0,\infty),\\
w(x,0)=w_0(x),
& x\in\mathbb{R}.
\end{cases}
\label{eqn:long-time}
\end{equation}
The effective Hamiltonian is then determined by the large-time averaging formula
\begin{equation}
\overline H(p) = -\lim_{t\to\infty}\frac{w(x,t)}{t}.
\label{eqn:long-time-formula}
\end{equation}
Since \eqref{eqn:long-time} is itself a QHJE of the form \eqref{eqn:QHJE},
its accurate solution provides a foundation for computing
quasiperiodic effective Hamiltonians through \eqref{eqn:long-time-formula}.

The theoretical analysis of QHJEs has been well-established through classical tools such as the vanishing viscosity
 method and variational principles \cite{lions1982generalized,evans2022partial,crandall1983viscosity}. These analytical tools provide a solid foundation for proving the existence, uniqueness, and stability of solutions to QHJEs with convex Hamiltonians.  In the periodic case, homogenization theory for Hamilton--Jacobi equations
was initiated by Lions, Papanicolaou, and Varadhan \cite{lions1987homogenization} and further developed in
the convex setting \cite{concordel1996periodic, concordel1997periodic, contreras1998lagrangian}. Based on these theoretical results, numerical
methods for computing periodic Hamilton--Jacobi equations and effective Hamiltonians have been studied extensively
\cite{qian2018min, falcone2008on, gomes2004computing, osher1988fronts,li2003numerical}. However, the research in the quasiperiodic case remains far less developed.  This limited development is mainly due to the numerical difficulties
associated with quasiperiodic solutions \cite{panov2017almost}, which lack translational
invariance and decay and may have low regularity. A natural approach is to use the periodic approximation method (PAM) to approximate the QHJE on a finite domain equipped with periodic boundary conditions. However,  this treatment cannot preserve the quasiperiodic nature of the solutions to QHJEs, and introduces Diophantine  errors arising from the approximation of irrational numbers by rational ones \cite{jiang2023approximation}.  

To avoid the Diophantine error, two accurate algorithms designed for quasiperiodic systems have been proposed recently. The first one is the projection method (PM) \cite{jiang2014numerical}. The PM embeds a low-dimensional quasiperiodic system into a
higher-dimensional periodic system, where the problem can be solved using
standard numerical discretizations. It has been successfully applied to the numerical
homogenization of
quasiperiodic elliptic equations \cite{jiang2025projection}. However, the PM is often combined with Fourier spectral collocation discretizations, which are highly effective for smooth quasiperiodic systems but
may be less suitable for QHJEs with nonsmooth solutions or low-regularity
gradients. To handle both high- and low-regularity quasiperiodic functions, another accurate algorithm,  the finite points recovery (FPR) method, has been proposed in  \cite{jiang2024accurately}. By exploiting the homomorphism between the physical domain and the irrational manifold on a high-dimensional torus defining the quasiperiodic structure, the FPR method  recovers the global quasiperiodic function by employing a local interpolation technique with finite points in the computational domain, making it well suited for low-regularity settings and systems involving multiple irrational numbers. However, FPR is a quasiperiodic function-recovery method. Solving a whole-space quasiperiodic model on a finite domain still requires an appropriate boundary condition that preserves the global quasiperiodic structure.

 In this paper, we formulate a computable QHJE and propose a high-accuracy numerical algorithm for its solution. Moreover, we apply the proposed method to quasiperiodic homogenization. Our contributions are summarized as follows.
\begin{itemize}
    
\item{We introduce a quasiperiodic boundary condition (QBC)
\cite{han2026accurately} for QHJEs, which preserves the quasiperiodic
structure and long-range order at the boundary.
Based on this boundary treatment, we obtain a computable reformulation
of the finite-size QHJE for accurate numerical approximation.}

\item{We propose a high-accuracy numerical algorithm, termed the SL--FPR scheme, which combines the semi-Lagrangian (SL) scheme with the FPR method for solving QHJEs. The scheme overcomes the difficulty of accurately evaluating quasiperiodic characteristic feet that fall outside the computational domain. Moreover, we provide a rigorous numerical analysis of the SL--FPR scheme.}

\item{We extend the quasiperiodic homogenization result for \eqref{eqn:k_hamiltonian} from the quadratic case \(k=2\) to the general case \(k>1\), and present two classes of numerical experiments for the accurate solution of QHJEs and the evaluation of effective Hamiltonians. The results verify the error analysis of the SL--FPR scheme and numerically confirm the extended homogenization result.}

\end{itemize}

The article is structured as follows. In \Cref{sec:preliminaries}, we provide an overview of quasiperiodic functions, along with a concise introduction to the SL scheme. In \Cref{sec:qbc}, we construct a computable formulation of the QHJE by restricting the original whole-space problem to a finite domain equipped with QBCs, and we present its numerical implementation using the FPR method. In \Cref{sec:numerical_methods}, we discretize the QHJE by the SL--FPR scheme and give the corresponding numerical analysis. In \Cref{sec:quasi_hom}, we extend the quasiperiodic homogenization theory for \eqref{eqn:k_hamiltonian} from the case $k=2$ to the general case $k>1$ and apply our algorithm to  accurately compute the effective Hamiltonians.
 In \Cref{sec:numerical_tests}, we present a series of numerical experiments to validate the accuracy of the SL--FPR scheme and the theoretical results.  In \Cref{sec:conclusion}, we summarize the conclusions and provide an outlook for future work.

\section{Preliminaries}\label{sec:preliminaries}
In this section, we give the definition of quasiperiodic functions and some of their necessary properties in \Cref{sec:pre_qpfunction}. Moreover, a brief introduction to the  semi-Lagrangian scheme is provided in \Cref{sec:pre_sls}. 
\subsection{Quasiperiodic functions}\label{sec:pre_qpfunction}
We first recall the definition of quasiperiodic functions and the properties needed in the subsequent analysis.
\begin{definition}\label{def:qp_function}
A matrix \(\bm{P}\in\mathbb{R}^{d\times n}\) is called a projection matrix if
\[
\bm{P}\in\mathbb{P}^{d\times n}
:=
\left\{
\bm{P}\in\mathbb{R}^{d\times n}:~
\operatorname{rank}_{\mathbb{Q}}(\bm{P})= \operatorname{rank}_{\mathbb{Q}}(\bm{p}_1, \cdots, \bm{p}_n) = n
\right\}.
\]
\end{definition}


\begin{definition}\label{def:QP_fun}
A continuous function \(f(\bm{x}) : \mathbb{R}^d \to \mathbb{R}\) is quasiperiodic 
if there exist an \(n\)-dimensional periodic function \(F\in C(\mathbb{T}^n)\) 
and a projection matrix \(\bm{P}\in \mathbb{P}^{d\times n}\) such that 
\begin{equation}\label{eqn:qp_homo}
    f(\bm{x})=F(\mathcal{C}(\bm{x})),\qquad 
    \mathcal{C}(\bm{x}) \coloneqq (\bm{P}^T\bm{x})/(2\pi\mathbb{Z}^n) \in \mathbb{T}^n.
\end{equation}
Here, \(F\) is called the parent function of \(f\), and
\(\mathcal C:\mathbb R^d\to\mathbb T^n\) is the physical-to-torus
homomorphism \cite{fan2025representation}, whose image forms the associated irrational manifold on
\(\mathbb T^n\). The set \(\mathrm{QP}(\mathbb{R}^d)\) consists of all
\(d\)-dimensional quasiperiodic functions.
\end{definition}

Let $K_T
=
\left\{
\bm{x}=(x_1,\ldots,x_d)\in\mathbb{R}^d:
|x_j|\leq T,\ j=1,\ldots,d
\right\}$.
For \(f\in\mathrm{QP}(\mathbb{R}^d)\), its mean value is defined by
\begin{equation}\label{eqn:elliptic:meanvalue}
\mathcal{M}\{f\}
:=
\lim_{T\to\infty}
\frac{1}{(2T)^d}
\int_{\bm{s}+K_T} f(\bm{x})\,d\bm{x},
\end{equation}
where the limit exists uniformly with respect to
\(\bm{s}\in\mathbb{R}^d\) \cite{jiang2024numerical}.

\begin{definition}
Let \(f\in\mathrm{QP}(\mathbb{R}^d)\). For
\(\bm{\lambda}\in\mathbb{R}^d\), its Bohr--Fourier coefficient is defined by
\[
\widehat f(\bm{\lambda})
:=
\mathcal{M}\!\left\{
f(\bm{x})e^{-\imath\bm{\lambda}^{T}\bm{x}}
\right\}.
\]
The Bohr spectrum of \(f\) is defined as
\[
\Lambda(f)
:=
\left\{
\bm{\lambda}\in\mathbb{R}^d:
\widehat f(\bm{\lambda})\neq0
\right\}.
\]
For \(r\geq0\), define
\[
|f|_{\mathcal H_{\rm QP}^{r}(\mathbb{R}^d)}^2
:=
\sum_{\bm{\lambda}\in\Lambda(f)\setminus\{\bm0\}}
|\bm{\lambda}|^{2r}
|\widehat f(\bm{\lambda})|^2.
\]
We write \(f\in\mathcal H_{\rm QP}^{r}(\mathbb{R}^d)\)
if this seminorm is finite.
\end{definition}


Based on these definitions, a fundamental property of quasiperiodic functions is given in \cite[Theorem~1.1]{fan2025representation} and summarized in the following lemma.

\begin{lemma}
	\label{lemma:fpr_homo}
	Let $\bm{P}\in\mathbb{P}^{d\times n}$.
	The $\mathbb{R}^d$-action on $\mathbb{T}^n$ determined by
	$\{\mathcal{C}({\bm{x}})\}_{\bm{x}\in\mathbb{R}^d}$
	is uniquely ergodic and minimal. Consequently,
	\begin{equation*}
		\overline{
			\left\{
			\mathcal{C}({\bm{x}})
			:
			\bm{x}\in\mathbb{R}^d
			\right\}}
		=
		\mathbb{T}^n.
	\end{equation*}
\end{lemma}
\begin{figure*}[!hbpt]
\centering
\includegraphics[width=12cm]{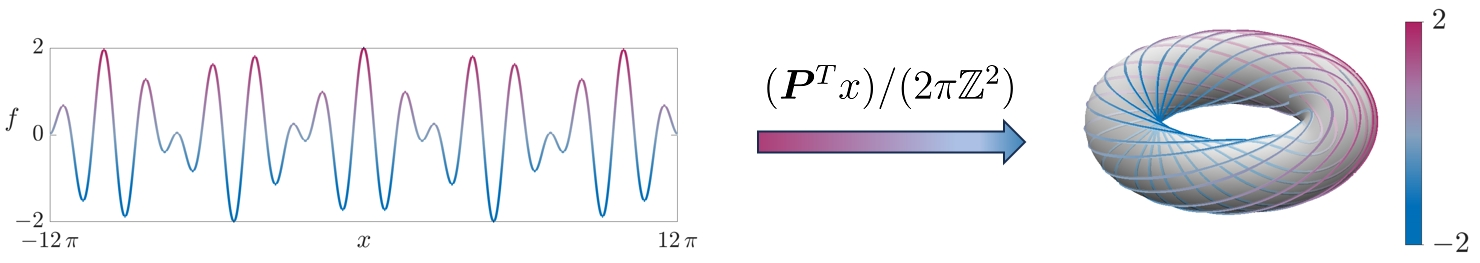}
\caption{Left: Quasiperiodic function $f(x)=\cos x + \cos(\sqrt{2}x)$ restricted to $[-12\pi,12\pi]$; Right: Corresponding parent function $F((\bm{P}^Tx)/(2\pi\mathbb{Z}^2)),\,x\in[-12\pi,12\pi]$.}
\label{fig:qp_fun}
\end{figure*}


We present a concrete example to demonstrate \Cref{lemma:fpr_homo}.
Consider the quasiperiodic function
\(f(x)=\cos x+\cos(\sqrt{2}\,x)\), defined on the entire real line.
For illustration, \Cref{fig:qp_fun} shows its restriction to
\([-12\pi,12\pi]\), together with the corresponding parent function
\(F(\mathcal{C}(x))\) on \(\mathbb{T}^2\).
As \(x\) varies, \(\mathcal C(x)\) traces the irrational orbit
\(\mathcal C(\mathbb R)\) in \(\mathbb T^2\).  By minimality, the full orbit is dense in
\(\mathbb{T}^2\), while unique ergodicity yields uniform distribution in
the infinite-domain limit. Consistent with these properties, the orbit
over the displayed finite interval already exhibits nearly uniform
coverage of \(\mathbb{T}^2\).



\subsection{Semi-Lagrangian scheme}\label{sec:pre_sls}
In this subsection, we  introduce the basic idea of the
semi-Lagrangian scheme for the Hamilton--Jacobi equation. First, we can reformulate the convex QHJE \eqref{eqn:QHJE} from a Lagrangian view.  To achieve this, we need to establish the relationship between the Hamiltonian and the Lagrangian, which leads us to the definition of the \emph{Legendre transform} \cite{tran2021hamilton}.
\begin{definition}\label{def:LT}
     For the Hamiltonian $H(x, p): \mathbb{R}\times\mathbb{R}\to\mathbb{R}$ which satisfies
    \begin{equation}\label{eqn:LT}
    \left\{\begin{array}{l}H \in C\left(\mathbb{R} \times \mathbb{R}\right), H \in \mathbf{BUC}\left(\mathbb{R} \times B(0, R)\right) \text { for each } R>0, \\ p \mapsto H(x, p) \text { is convex for all } x \in \mathbb{R}, \\ H \text { is superlinear in } p ; \text { that is, } \lim\limits_{|p| \rightarrow \infty}  \dfrac{H(x, p)}{|p|}=+\infty, \end{array}\right.    
    \end{equation}
     we define its Legendre transform $H^*=L: \mathbb{R}\times\mathbb{R}\to\mathbb{R}$ as
     \begin{equation}
        H^*(x, a) =  L(x, a)=\sup _{p \in \mathbb{R}}(p \cdot a-H(x, p)).
     \end{equation}
\end{definition}

By \Cref{def:LT}, 
we  can compute the Lagrangian \(L\) of the Hamiltonian \eqref{eqn:k_hamiltonian}. 
The result is stated in the following theorem (see also \cite[Chapter~2]{tran2021hamilton}).

\begin{theorem}
Let \(k>1\), and consider the Hamiltonian
\[
H(x, p)=\frac{1}{k}|p|^k - f(x),\qquad (x, p)\in\mathbb{R}\times\mathbb{R}.
\]
  Then the Legendre transform of \(H\)  is
\[
L(x,a)=\sup_{p\in\mathbb{R}}\big\{p\cdot a - H(x, p)\big\}
= \frac{1}{k'}\,|a|^{k'} + f(x),
\qquad k'=\frac{k}{k-1}.
\]
\end{theorem}
Using the Legendre transform, the QHJE \eqref{eqn:QHJE} can be reformulated as the well-known quasiperiodic Hamilton--Jacobi--Bellman equation (QHJBE) in optimal control theory.

\begin{equation}\label{eqn:QHJBE}
\left\{\begin{aligned}
        &u_t-\min_{a\in\mathbb{R}}\left\{-a u_x +L(x, a)\right\}=0,\quad \mathrm{in}\quad \mathbb{R}\times (0,T), \\
        &u(x, 0) = u_0(x), \quad\mathrm{on}~~ \mathbb{R}\times\left\{t=0\right\}.
\end{aligned}
    \right.
\end{equation}
\begin{remark}\label{rem:min}
   For each fixed $(x, t)$,  the strict convexity of \(L\) in the QHJBE \eqref{eqn:QHJBE} guarantees
a unique minimizer \(\bar a\). The first-order optimality condition yields
\[
\bar a=\operatorname{sgn}(p)|p|^{\frac{1}{k'-1}}, \qquad p =u_x(x, t).
\]
\end{remark}

To discretize \eqref{eqn:QHJBE}, we apply the SL scheme, which is known to be stable even for large time steps and to converge to the viscosity solution of the HJ equation provided that the solution is unique \cite{falcone2013semi}. To establish the approximation scheme, we first introduce a uniform grid in space and time  with constant steps $h$ and $\tau$, respectively, covering the domain
 \begin{equation*}
     \left\{(x_i, t_m) = (ih, m\tau), ~i, m\in\mathbb{N},~ i\leq \frac{|\Omega|}{h} = D, ~m\leq\frac{T}{\tau} = M\right\}.
 \end{equation*}
 
The SL  scheme  approximates the advection
 term as a directional derivative
 $$
au_x\approx-\frac{u^m(x_i-a\tau) - u(x_i, t_m)}{\tau},
 $$
and the time derivative is discretized by the forward Euler method
\begin{equation*}
    u_t(x_i,t_m)
    \approx
    \frac{u_i^{m+1}-u_i^m}{\tau}.
\end{equation*}
Substituting these approximations into \eqref{eqn:QHJBE}, we have
\begin{equation}\label{eqn:sls_QHJE}
    u_i^{m+1}
    =
    \min_{a\in\mathbb R}
    \left\{
        u^m(x_i-a\tau)
        +\tau L(x_i,a)
    \right\}.
\end{equation}


Here, $u^{m}(x_i-a\tau)$
denotes the solution at the \(m\)-th time level evaluated at the
characteristic foot \(x_i-a\tau\), obtained by tracing backward from
\(x_i\) along the characteristic direction \(a\) over one time step
\(\tau\).

\begin{remark}
At each time step, the SL scheme requires solving a one-dimensional
minimization problem in \(a\). Since the objective is continuous and
coercive, a global minimizer exists \cite[Theorem~1.9]{rockafellar1998variational}.
In practice, the analytical control from \Cref{rem:min} is used to center
the search interval, within which the SL objective is minimized numerically up to the prescribed tolerance.
\end{remark}

Since the position $x_i-a\tau$ appearing in \eqref{eqn:sls_QHJE} may not align exactly with the grid nodes, as illustrated in \Cref{fig:SL_trace}, it is necessary to compute the term  $u^{m}(x_i-a\tau)$ using reconstruction techniques.
\begin{figure}[!htbp]
    \centering
    \includegraphics[width=0.6\linewidth]{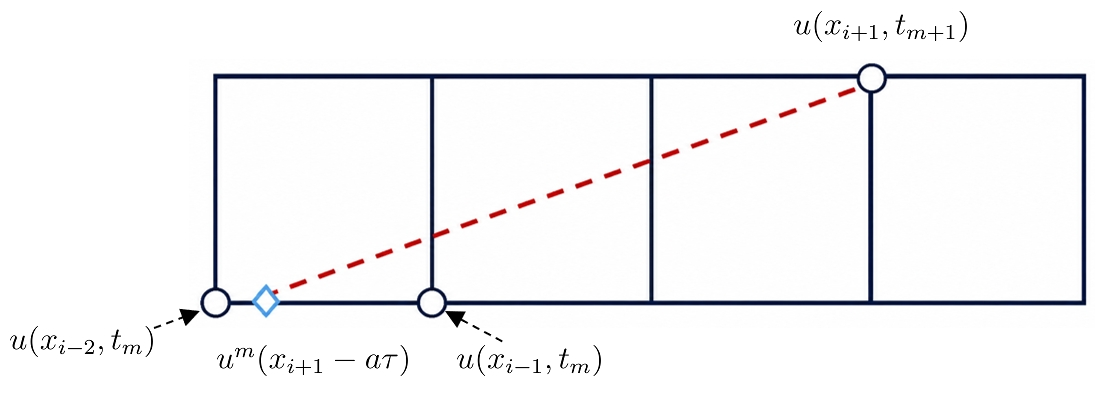}
    \caption{The trace-back process in the SL scheme.}
    \label{fig:SL_trace}
\end{figure}
 
 Moreover, since $x_i-a\tau$ may lie outside the
bounded computational domain \(\Omega\), appropriate boundary conditions must be imposed to restrict the computation to a finite domain. For the periodic case, the  implementation of periodic boundary conditions (PBCs) is simple since we just
 set
 $$
u(x+T)=u(x),
 $$
where $T$ is the period of the solution. However, this treatment is not directly applicable in the quasiperiodic
setting because periodic approximation introduces a Diophantine error through rational approximation of the irrational numbers \cite{jiang2023approximation}.
To address this difficulty, we propose a QBC for the QHJE, under which the numerical solution at points outside $\Omega$ (including $ \partial \Omega$) along the characteristics of \eqref{eqn:QHJE} can be recovered from a finite set of data points in $\Omega$, with the reconstruction realized through the 
FPR method.

\section{Quasiperiodic boundary conditions for QHJEs}\label{sec:qbc}
In this section, we develop a finite-domain formulation for QHJEs with the \(k\)-power convex Hamiltonian $H(x,p)={|p|^k}/{k}-f(x), ~k>1$. 
We first construct a Dirichlet-type QBC that preserves the quasiperiodic structure of the original problem after domain truncation. We then present the FPR method for its numerical implementation.

\subsection{Computable modeling for QHJEs}

As discussed in \Cref{sec:pre_sls}, the SL discretization for QHJEs
requires an appropriate boundary treatment on a bounded computational
domain. To establish a computable finite-domain model for QHJEs while preserving the quasiperiodic structure of the original problem, the
boundary values should remain consistent with the interior field. To
achieve this, we use the homomorphism \(\mathcal C\) introduced in
\Cref{sec:pre_qpfunction}, which associates the physical space with an
irrational manifold embedded in the high-dimensional torus. This
requirement leads to the following definition.

\begin{definition}
\label{def:qbc}
Let \(u\in \mathrm{QP}(\mathbb{R}^d)\)~with associated parent function
\(U\) and projection matrix \(\bm P\in\mathbb P^{d\times n}\), and let
\(\Omega\subset\mathbb R^d\) be a bounded domain.
 Using the physical-to-torus homomorphism \(\mathcal{C}\), we define
the QBC by
\begin{equation}\label{eq:QBC_def}
u(\bm{x})
=
U\bigl(\mathcal{C}(\bm{x})\bigr),
\qquad
\bm{x}\in\partial\Omega.
\end{equation}
\end{definition}

The QBC \eqref{eq:QBC_def} inherits its boundary data from the same
parent function that generates the interior field, thereby preserving
the global quasiperiodic structure after domain truncation. At the
continuous level, \eqref{eq:QBC_def} is an exact identity. In numerical
implementation, the parent-function values are recovered from finite
interior samples by the FPR method, so the approximation arises only
from this recovery, as detailed in \Cref{sec:fpr_qbc}.



\Cref{fig:boundary1_on_torus} illustrates the QBC for the
one-dimensional domain \(\Omega=[-7\pi,6.25\pi]\). Domain truncation
creates two boundary points (red dots), whose phases remain on the same irrational
manifold in \(\mathbb{T}^{2}\). Their boundary values are therefore
determined by the parent function \(U\).

\begin{figure*}[!hbpt]
\centering
\subfigure[]{\includegraphics[width=4.5cm]{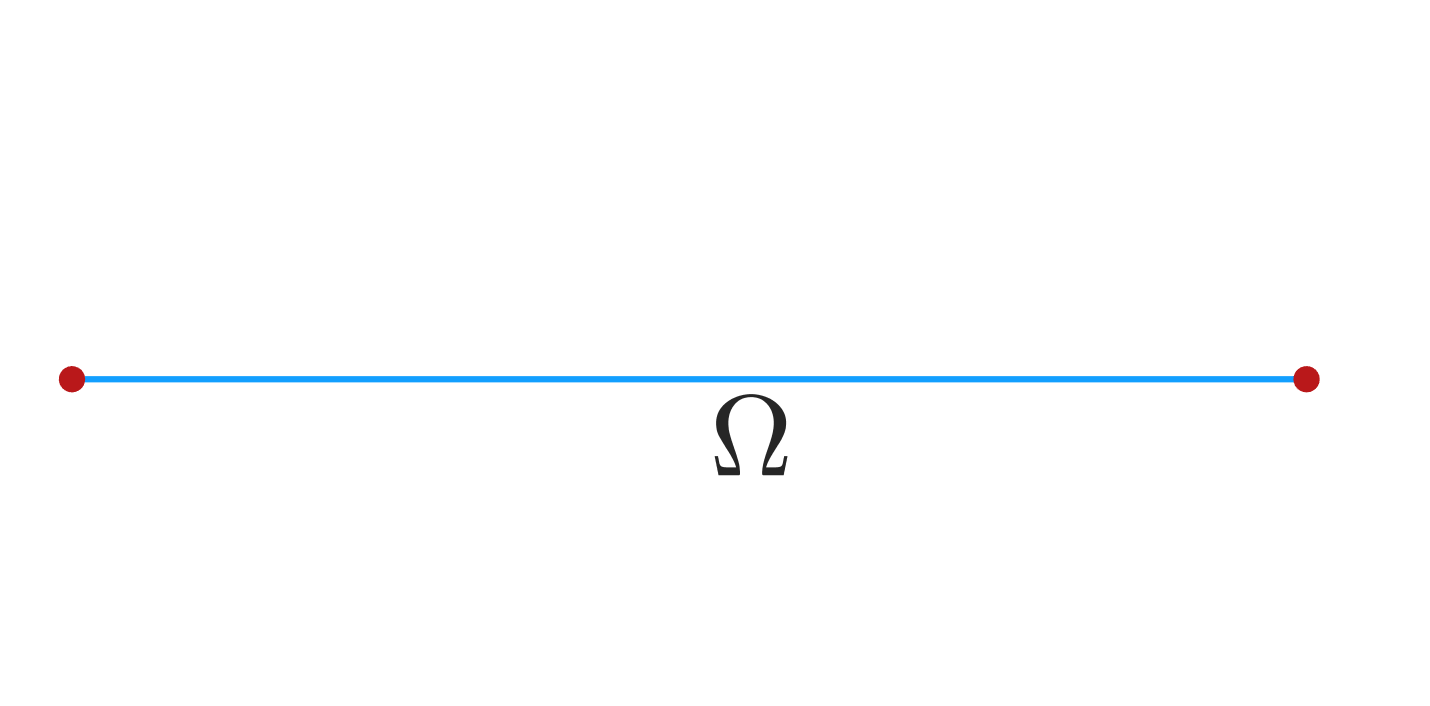}}
\hspace{5mm}
\subfigure[]{\includegraphics[width=4cm]{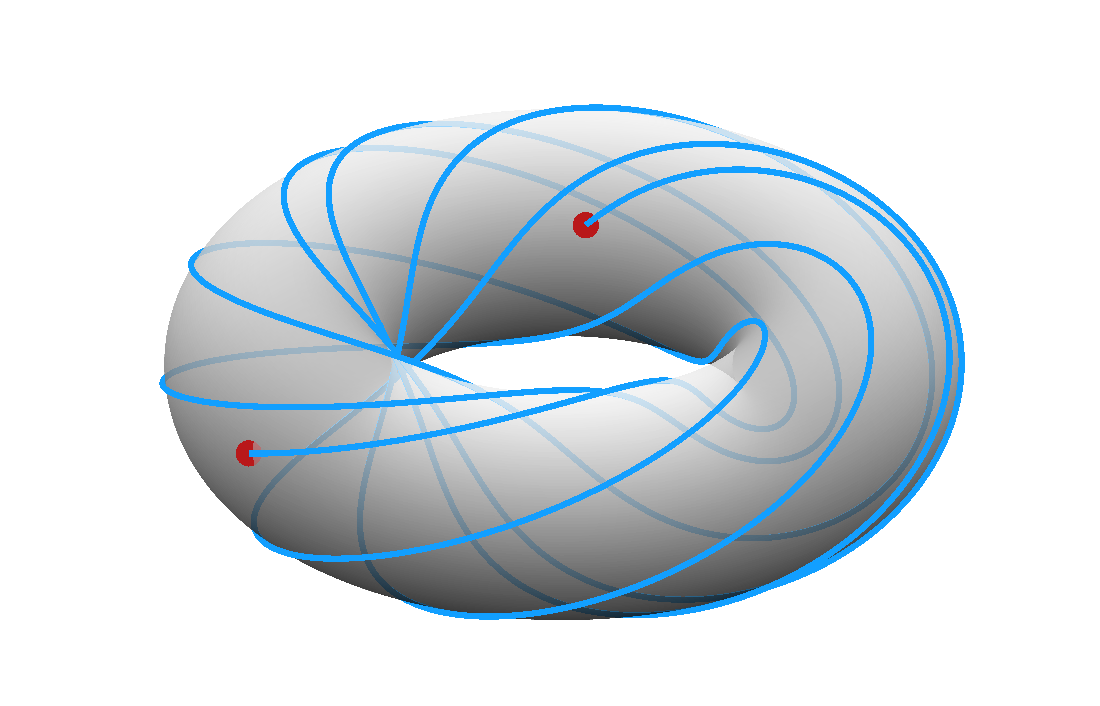}}
\caption{
 Illustration of the QBC for
\(\Omega=[-7\pi,6.25\pi]\) with
\(\bm{P}=(1,\sqrt{2})\):
\textup{(a)} the bounded domain and
\textup{(b)} its corresponding images on
\(\mathbb{T}^{2}\).
}
\label{fig:boundary1_on_torus}
\end{figure*}



Building on the QBC, we formulate the following computable QHJEs
 \begin{equation}\label{eqn:QHJE1}
\left\{
               \begin{aligned}
                        &u_t+H(x, u_x)=0, \quad\mathrm{in}~~ \Omega\times(0,\infty),\\
                        &u(x, 0)=u_0(x),\quad\mathrm{on}~~ \Omega\times\left\{t=0\right\},\\
                        & u(x, t) =U(\mathcal{C}(x), t), \quad \mathrm{on}~~\partial\Omega\times(0,\infty).
                \end{aligned}                               
                \right.
\end{equation}
Here, \(H\) takes the form \eqref{eqn:k_hamiltonian} and satisfies
assumptions \ref{itm:quasiperiodic}--\ref{itm:convex}.~Together with
$u_0\in\mathbf{BUC}(\mathbb{R})\cap\mathbf{Lip}(\mathbb{R})$, these
assumptions ensure the existence of a unique viscosity solution;
see \cite[Theorem~1.38]{tran2021hamilton}.

\subsection{Approximation of the QBC}\label{sec:fpr_qbc}

In this subsection, we introduce the FPR method for implementing the QBC.
The method exploits the physical-to-torus homomorphism
\(\mathcal C\).
By \Cref{lemma:fpr_homo}, the image
\(\mathcal C(\mathbb R^d)\) is dense in \(\mathbb T^n\). Hence, for a target point \(\bm x^*\), one can select physical points
whose mapped phases are close to \(\mathcal C(\bm x^*)\) and recover
the corresponding parent-function value \(U(\mathcal C(\bm x^*))\) by local interpolation.
This avoids rational approximation of the irrational numbers.
The procedure consists of the following two steps.

\begin{itemize}
    \item \textbf{Selection of interpolation nodes.}
    Choose a domain \(\Omega\subset\mathbb{R}^d\) and
construct a physical grid \(\Omega_N\subset\Omega\). The mapped point set $\mathcal{C}(\Omega_N)
    :=
    \left\{\mathcal{C}(\bm{x}_j):\bm{x}_j\in\Omega_N\right\}$
is a subset of the irrational manifold $(\bm{P}^T\bm{x})/(2\pi\mathbb{Z}^n)$
and sufficiently samples \(\mathbb{T}^n\).
    For a target point \(\bm{x}^{*}\in\partial\Omega\), compute its phase
    \(\mathcal{C}(\bm{x}^{*})\) and select
    \(K\) grid points $\Theta_K
        :=
        \left\{\bm{x}_{j_\ell}\right\}_{\ell=1}^{K}
        \subset\Omega_N$, 
    whose mapped phases are close to
    \(\mathcal{C}(\bm{x}^{*})\) in \(\mathbb{T}^n\).

    \item \textbf{Interpolation recovery.}
    Construct the local  basis functions
    \(\{\phi_\ell\}_{\ell=1}^{K}\) from the selected mapped phases.
    The value at \(\bm{x}^{*}\) is then recovered by
    \begin{equation*}
        u(\bm{x}^{*})
        =
        U\bigl(\mathcal{C}(\bm{x}^{*})\bigr)
        \approx
        \mathcal{I}_{\mathrm{FPR}}u(\bm{x}^{*})
        :=
        \sum_{\ell=1}^{K}
        u\bigl(\bm{x}_{j_\ell}\bigr)
        \phi_\ell\bigl(\mathcal{C}(\bm{x}^{*})\bigr),
    \end{equation*}
    where \(\mathcal{I}_{\mathrm{FPR}}\) denotes the 
    FPR interpolation operator.
\end{itemize}

When degree-$s$ Lagrange interpolation is employed, the FPR interpolation
operator is denoted by $\mathcal{I}_{\mathrm{FPR}(s)}$.  The following
theorem gives the corresponding  error estimate.

\begin{theorem}
\label{lem:FPR_error}
Let $u\in\mathcal{H}_{\rm QP}^{r}(\mathbb{R}^{d}),~ r>s+1+{n}/{2}$.
Assume that there exists a constant $c_{\bm P}>0$ such that
\begin{equation}\label{eqn:spec_spera}
|\bm P\bm m|
\geq
c_{\bm P}|\bm m|,
\qquad
\bm m\in\mathbb{Z}^{n}\setminus\{\bm 0\},
\quad
\bm P\bm m\in\Lambda(u),
\end{equation}
where $\bm{P}$ is the projection matrix of $u$. Then, for the $s$-th FPR interpolation, we have
\[
\|u-\mathcal{I}_{\mathrm{FPR}(s)}u\|_{\infty}
\leq
C h^{s+1}
|u|_{\mathcal{H}_{\rm QP}^{r}(\mathbb{R}^{d})},
\qquad
s\in\mathbb N^{+},
\]
where $C$ depends on $c_{\bm P}$.

\begin{proof}
Let \(U\in C(\mathbb{T}^n)\) be the parent function of \(u\). By the relation between the Fourier coefficients of \(u\) and those of its parent function \(U\) \cite[Theorem~4.1]{jiang2024numerical}, we have
\[
\widehat{U}(\bm m)=\widehat{u}(\bm P\bm m).
\]
Then, using the assumption \eqref{eqn:spec_spera},
we obtain
\[
|U|_{H^{r}(\mathbb T^{n})}^{2}
=
\sum_{\bm m\in\mathbb Z^{n}\setminus\{\bm 0\}}
|\bm m|^{2r}|\widehat U(\bm m)|^{2}
\leq
c_{\bm P}^{-2r}
|u|_{\mathcal H_{\rm QP}^{r}(\mathbb R^{d})}^{2}.
\]
Since \(r>s+1+n/2\), the Sobolev embedding on \(\mathbb T^n\) \cite[Section 5.1]{einsiedler2017functional} gives
\begin{equation}\label{eqn:fpr_estimate1}
|U|_{W^{s+1,\infty}(\mathbb T^n)}
\lesssim
|U|_{H^r(\mathbb T^n)}
\lesssim
c_{\bm P}^{-r}
|u|_{\mathcal H_{\rm QP}^{r}(\mathbb R^d)}.    
\end{equation}
By the degree-\(s\) FPR interpolation estimate
\cite[Theorem~4.4]{jiang2024accurately},
\begin{equation}\label{eqn:fpr_estimate2}
\|u-\mathcal I_{\mathrm{FPR}(s)}u\|_{\infty}
\lesssim
 h^{s+1}
|U|_{W^{s+1,\infty}(\mathbb T^{n})}.    
\end{equation}
Combining the estimates \eqref{eqn:fpr_estimate1} and \eqref{eqn:fpr_estimate2} yields
\[
\|u-\mathcal I_{\mathrm{FPR}(s)}u\|_{\infty}
\leq
C h^{s+1}
|u|_{\mathcal H_{\rm QP}^{r}(\mathbb R^{d})}.
\]
\end{proof}
\end{theorem}

The FPR method performs computations within the physical space and
thus avoids dimensional lifting, overcoming the high computational cost
caused by the curse of dimensionality compared with the PM.
Moreover, since the FPR method relies on local interpolation recovery, it is naturally applicable to problems with low regularity. In the numerical
implementation of the QBC through the FPR method, the resulting  error is solely determined by the FPR interpolation error, which can be reduced by mesh refinement or higher-order interpolation.
\section{Discretization for QHJEs and convergence analysis}
\label{sec:numerical_methods}

In this section, we propose a new approach for QHJEs \eqref{eqn:QHJE1} that combines the SL scheme with FPR interpolation, referred to as the SL–FPR scheme. We also present the stability and convergence analysis of the scheme.

\subsection{SL--FPR scheme}
The SL--FPR  scheme for the QHJE \eqref{eqn:QHJE1} consists of two main steps. The first step applies the SL scheme to discretize  \eqref{eqn:QHJE1} as 
\begin{equation}\label{eqn:sls_QHJE2}
     u^{m+1}_i=\min_{a\in\mathbb{R}}\left\{u^m(x_i-a\tau)+\tau(\frac{1}{k'}a^{k'}+f(x_i))\right\}.
 \end{equation}

The second step uses the FPR method to reconstruct the quasiperiodic characteristic foot $u^m(x_i-a\tau)$ within the QBC. We rewrite the discrete scheme \eqref{eqn:sls_QHJE2} as
\begin{equation}\label{eqn:sls_fpr}
    \begin{aligned}
        u^{m+1}_i &= \min_{a\in\mathbb{R}}\left\{u^m(x_i-a\tau)+\tau \big(\frac{1}{k'}a^{k'}+f(x_i)\big)\right\}\\
        &\approx  \min_{a\in\mathbb{R}}\left\{\mathcal{I}_{\mathrm{FPR}(s)}[u^{m}](x_i-a\tau)+\tau \big(\frac{1}{k'}a^{k'}+f(x_i) \big)\right\}\\
        &= 
        \mathcal{I}_{\mathrm{FPR}(s)}[u^{m}]
        (x_i-\overline{a}_i^m\tau)
        +\tau(\frac{1}{k'}(\overline{a}_i^m)^{k'}
        +f(x_i)).
    \end{aligned}
\end{equation}
Here, $\mathcal{I}_{\mathrm{FPR}(s)}[u^{m}](x_i-\overline{a}^m_i\tau)$ denotes the recovery of the value $u^m(x_i-a\tau)$  using the FPR method, where $\overline{a}_i^m$ denotes a minimizer of the discrete objective function.   For convenience, we write $\mathcal{I}_{\mathrm{FPR}(s)}:=\mathcal{I}_s$ throughout the following analysis.

To more clearly introduce the basic idea of the SL–FPR scheme, we present an illustration of the SL--FPR(1) scheme as a representative  example. As shown in  \Cref{fig:sl_boundary1_on_torus}, the left panel illustrates the characteristic tracing and reconstruction procedures in the SL scheme. 
The red point at the time level \(t_{m+1}\) denotes the target value. Its SL update requires the solution value at the characteristic foot traced back to \(t_m\), which is reconstructed from the four orange points in \(\Theta_4\) using first-order FPR interpolation.
The right panel shows how the set $\Theta_4$  is determined via the homomorphic mapping $\mathcal{C} (\Omega_N) = (\bm{P}^T\Omega_N)/(2\pi\mathbb{Z}^2)$, which selects the orange points lying closest to the target point on $\mathbb{T}^n$.
\begin{figure*}[!htbp]
\centering
{
    \includegraphics[width=10cm]{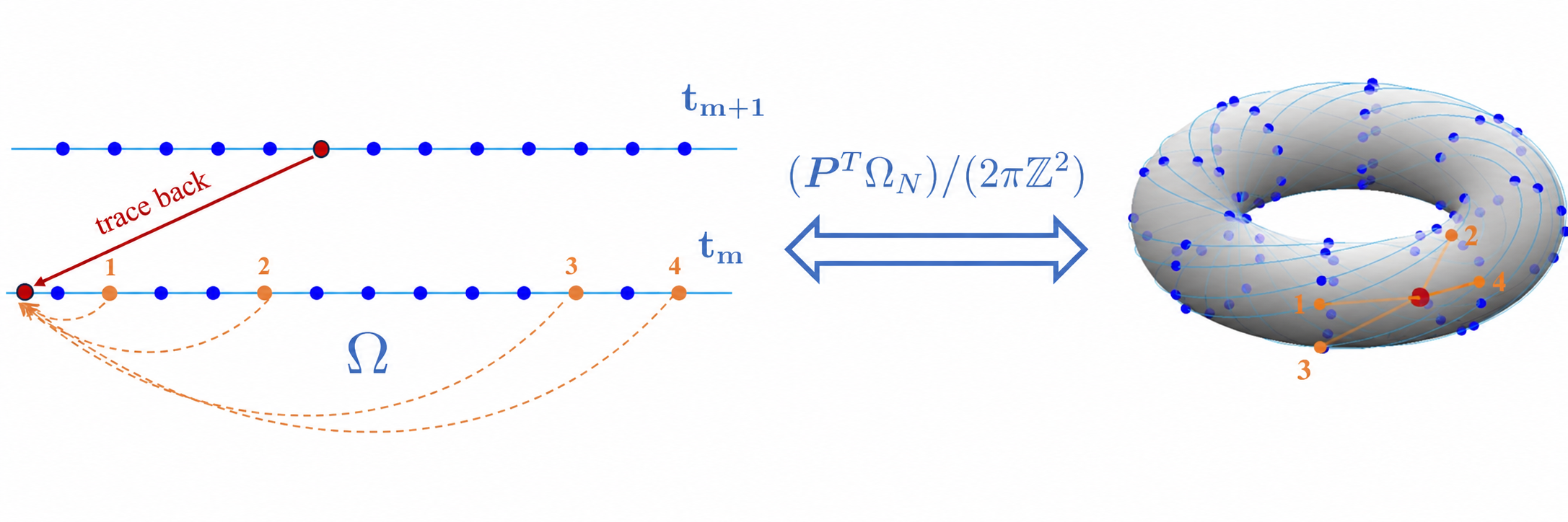}
}
\caption{Illustration of the SL--FPR(1) scheme for recovering characteristic foot.}
\label{fig:sl_boundary1_on_torus}
\end{figure*}

\subsection{Stability}  Let us consider the  stability of the scheme. Since the minimization is performed over \(a\in\mathbb R\), the minimizer may be unbounded, which poses a potential issue in both the analysis and numerical computation. Fortunately, the following lemma provides a method for identifying an equivalent compact set that contains the minimizer in the convex setting, thereby addressing this challenge.


\begin{lemma}[{\cite[Proposition 5.1]{corrias1995numerical}}]\label{lemma:compect_set}
     Let $H: \mathbb{R}\times\mathbb{R}\to\mathbb{R}$  be continuous and convex. Moreover, let $H$ satisfy \eqref{eqn:LT}. Then, there exists a bounded set $A\subset\mathbb{R}$ such that
     $$
     H(p) = \sup_{a\in\mathbb{R}}\left\{p\cdot a-L(x, a)\right\} = \sup_{a\in A}\left\{p\cdot a-L(x, a)\right\}.
     $$
\end{lemma}

Based on this, we establish an \(L^\infty\) stability estimate for the SL scheme. 
\begin{theorem}
    The solution to the QHJE \eqref{eqn:QHJE1} under the SL scheme is stable with respect to the quasiperiodic potential term $f$ and the initial condition $u_0$.

    \begin{proof} According to \Cref{lemma:compect_set}, there exists a uniform constant $C_a>0$ such that 
\[
|\overline{a}_j| \le C_a \quad \text{for all } ~~1\leq j\leq m
\]
in the scheme \eqref{eqn:sls_QHJE}. Here, $\overline{a}_j$ denotes the minimizer of the SL scheme at each time level $j$. Thus, we have
    \begin{equation}
    \begin{aligned}
\|u^{m}\|_{\infty}&=\Big\|\min_{a\in\mathbb{R}}\left\{u^{m-1}(x-a\tau)+\tau(\frac{1}{k'}a^{k'}+f(x))\right\}\Big\|_{\infty}\\
&=\|u^{m-1}(x-\overline{a}_m\tau)+\tau(\frac{1}{k'}\overline{a}_m^{k'}+f(x))\|_{\infty}\\
&\leq \|u^{m-1}\|_\infty+\tau(\frac{1}{k'}\overline{a}_m^{k'}+\|f\|_{\infty})\\
&\leq\|u^{m-2}\|_{\infty}+\tau(\frac{1}{k'}(\overline{a}_m^{k'}+\overline{a}_{m-1}^{k'})+2\|f\|_{\infty})\\
&\leq\|u_0\|_\infty+m\tau\left(\frac{C_a^{k'}}{k'}+\|f\|_\infty\right).
\end{aligned}
\end{equation}
    \end{proof}
\end{theorem}

\textbf{Convergence analysis.}~~For the exact solution \(u\), define the one-step SL  approximation at
\((x_i,t_{m+1})\) by
\begin{equation}
\label{eqn:SL_operator}
\mathcal S[u(\cdot,t_{m})](x_i)
:=
\min_{a\in\mathbb R}
\left\{
u(x_i-a\tau,t_m)+\tau L(x_i,a)
\right\},
\end{equation}
and the numerical solution as
\begin{equation}
\label{eqn:FPR_solution}
u_i^{m+1}
:=
\min_{b\in\mathbb R}
\left\{
\mathcal I_s[u^m](x_i-b\tau)
+\tau L(x_i,b)
\right\},
\end{equation}
where \(a\) and \(b\) denote the optimization variables of the SL and SL--FPR
objective functions, respectively. Moreover, we denote by $
\mathcal S^{m+1}[u](x_i)$
the SL approximation at \(t_{m+1}\) obtained after \(m+1\) successive time steps. Then we  give the  error estimate of the SL--FPR$(s)$ scheme.



\begin{theorem}\label{thm:convergence}
  Assume that the exact solution satisfies
\(u(\cdot,t)\in\mathcal H_{\rm QP}^{r}(\mathbb R)\) with
\(r>s+1+n/2\) ~for~ \(0\le t\le T\), that \(u\in C^2(\mathbb R\times[0,T])\), and that the spectral-separation condition
\eqref{eqn:spec_spera} holds. Then, the SL--FPR$(s)$ scheme satisfies the error estimate
    \begin{equation}
        \|u(x_i, t_{m+1}) - u^{m+1}_i\|_\infty\leq\mathcal{O}(h^{s+1}) + \mathcal{O}(\tau), \quad s\geq 1.
    \end{equation}
    \begin{proof}
     To quantify the spatial error, we consider the estimate
    $$
    | \mathcal{S}[u(\cdot,t_{m})](x_i) - u^{m+1}_i| = |\tau L(\overline{a}) + u^m(x_i-\overline{a}\tau)-(\tau L(\overline{b})+\mathcal{I}_su^m(x_i-\overline{b}\tau))|.
    $$
    To simplify notation and clarify the structure, we define
$$
\Phi(a):=\tau\,L(x_i,a)+u(x_i-a\tau,t_m),\qquad
\Psi(b):=\tau\,L(x_i,b)+\mathcal I_s u^m(x_i-b\tau).
$$
Then the SL approximation and the SL--FPR solution can be expressed as
$$ \mathcal{S}[u(\cdot,t_{m})](x_i)=\min\limits_{a\in\mathbb{R}}\Phi(a),\quad u^{m+1}_i=\min\limits_{b\in\mathbb{R}}\Psi(b).$$
Choose $\bar a$ (resp.\ $\bar b$) realizing (or approximating) the minimum of $\Phi$ (resp.\ \(\Psi\)), and set 
 $$y_{\bar a}=x_i-\bar a\tau,\quad y_{\bar b}=x_i-\bar b\tau.$$ Then
$$
\begin{aligned}
 \mathcal{S}[u(\cdot,t_{m})](x_i) - u^{m+1}_i
&= \Phi(\bar a) - u^{m+1}_i
\ge \Phi(\bar a)-\Psi(\bar a)
= u(y_{\bar a},t_m)-\mathcal I_s u(y_{\bar a},t_m),
\end{aligned}
$$
since \(u^{m+1}_i=\min\limits_{a\in\mathbb{R}}\Psi(a)\le\Psi(\bar a)\). Similarly,
\[
\begin{aligned}
 \mathcal{S}[u(\cdot,t_{m})](x_i) - u^{m+1}_i
&\le \Phi(\bar b)-\Psi(\bar b)
= u(y_{\bar b},t_m)-\mathcal I_s u(y_{\bar b},t_m),
\end{aligned}
\]
because \( \mathcal{S}[u(\cdot,t_{m})](x_i)=\min\limits_{b\in\mathbb{R}}\Phi(b)\le\Phi(\bar b)\). Combining the two inequalities yields
   \begin{equation}
\begin{aligned}
\| \mathcal{S}[u(\cdot,t_{m})](x_i)-u_i^{m+1}\|_\infty
&\leq
\max\left\{
\begin{array}{l}
\left|u^m(x_i-\bar a\tau)
-\mathcal{I}_s u^m(x_i-\bar a\tau)\right|,\\[2mm]
\left|u^m(x_i-\bar b\tau)
-\mathcal{I}_s u^m(x_i-\bar b\tau)\right|
\end{array}
\right\} \\
&\leq \|u-\mathcal{I}_s u\|_\infty
\leq C h^{s+1}.
\end{aligned}
\end{equation}


For the temporal discretization error, a Taylor expansion gives
\[
u(x-a\tau,t)
= u(x,t) - \tau\, a\cdot u_x(x,t) + \mathcal{O}(\tau^2).
\]
Substituting into the  SL scheme update formula
\[
\mathcal S[u(\cdot,t)](x) =  \min_{a\in\mathbb{R}}\big\{u(x-a\tau,t) + \tau\,L(x,a)\big\},
\]
the leading $a$-dependent term in the minimization is 
$L(x,a) - a\cdot D u(x,t)$, 
which attains its minimum at $\overline{a}$ satisfying $$\frac{\partial L(x, {a})}{\partial a} = \frac{\partial u(x,t)}{\partial x}.$$  
By the Legendre transform, $L(x, \overline{a}) - \overline{a}\cdot u_x(x,t) = -H\big(x,u_x\big)$, 
so
\[
\mathcal S[u(\cdot,t)](x)
= u(x,t) - \tau\, H\big(x,u_x\big) + \mathcal{O}(\tau^2).
\]
On the other hand, performing a Taylor expansion of the exact solution \(u(x, t+\tau)\) at time \(t\) yields

\[
u(x,t+\tau) 
= u(x,t) + \tau\,u_t(x,t) + \mathcal{O}(\tau^2)
= u(x,t) - \tau\, H\big(x, u_x(x,t)\big) + \mathcal{O}(\tau^2),
\]
where we used $u_t = -H(x,u_x)$. 
Subtracting these expressions, the local truncation error of the SL time discretization is
\[
u(x,t+\tau) - \mathcal S[u(\cdot,t)](x) = \mathcal{O}(\tau^2).
\]

Using the stability of the SL scheme, accumulation of the local
truncation errors over a fixed time interval yields
\[
\begin{aligned}
\left\|
u(\cdot,t_{m+1})
-
\mathcal S^{m+1}[u](\cdot)
\right\|_\infty
&\leq
\sum_{j=0}^{m}
\left\|
u(\cdot,t_{j+1})
-
\mathcal S[u(\cdot,t_j)](\cdot)
\right\|_\infty
\\
&\leq
C(m+1)\tau^2
=
\mathcal O(\tau).
\end{aligned}
\]


In conclusion, the total error of the SL--FPR(\(s\)) scheme satisfies
\begin{align*}
\left\|
u(x_i,t_{m+1})-u_i^{m+1}
\right\|_\infty
&\leq
\left\|
u(x_i,t_{m+1})
-\mathcal S^{m+1}[u](x_i)
\right\|_\infty
\\
&\quad+
\left\|
\mathcal S^{m+1}[u](x_i)
-u_i^{m+1}
\right\|_\infty
\\
&\leq
\mathcal{O}(\tau)
+
\sum_{j=0}^{m}
\left\|
\mathcal{S}[u(\cdot,t_{j})](x_i)-u_i^{j+1}
\right\|_\infty
\\
&=
\mathcal{O}(\tau)
+
\mathcal{O}(h^{s+1}).
\end{align*}


\end{proof}
    
\end{theorem}

\section{Application to quasiperiodic homogenization} \label{sec:quasi_hom}

In this section, we apply the SL--FPR scheme to compute the effective Hamiltonians in quasiperiodic homogenization. We consider the corrector problem 
\begin{equation}
H(x,p+v')
:=
\frac{1}{k}|p+v'|^k-f(x)
=
\overline H(p),
\qquad x\in\mathbb{R}.
\label{eq:corrector_problem}
\end{equation}  
associated with the \(k\)-power convex Hamiltonian
\[
H(x,p)=\frac{|p|^k}{k}-f(x),
\qquad k>1,
\]
where \(f(x)\) is quasiperiodic and
\(\overline H(p)\) denotes the effective Hamiltonian.
For mechanical Hamiltonians, the exact corrector problem has been
established in the quadratic case \(k=2\)
\cite{hu2024polynomial}. Here, we first extend this result to the general case $k>1$.

\begin{theorem}\label{thm:keh}
Let the Hamiltonian be
\[
H(x, p) = \frac{|p|^k}{k} - f(x),\quad k>1
\]
with $f(x) = F(\mathcal{C}(x))$ and $\mathcal{C}(x) = (\bm{P}^T\bm{x})/(2\pi\mathbb{Z}^n)$. Assume that
 $\min\limits_{\bm{y}\in\mathbb{T}^n} F(\bm{y}) = 0$.
Then  $\overline{H}(p)$ satisfies 

\begin{enumerate}
    \item[(i)] Define
    \[
    p_0 = \int_{\mathbb{T}^n} \left( k F(\bm{y}) \right)^{1/k}\,d\bm{y} = \ave{(k F(\bm{y}))^{1/k}}.
    \]
    Then $\overline{H}(p) = 0$ for $|p| \leq p_0$, and for $|p| \geq p_0$,  $\overline{H}(p)$ is determined by
    \[
    |p| = \int_{\mathbb{T}^n} \left( k (\overline{H}(p) + F(\bm{y})) \right)^{1/k}  d\bm{y}.
    \]

    \item[(ii)] 
     If \(|p|\ge p_0\), \(\bm P\) satisfies the Diophantine condition
 \begin{equation}\label{eqn:dio_condition}
 |\bm P\bm m|
 \geq
 \gamma |\bm m|^{-\tau},
 \qquad
 \bm m\in\mathbb Z^n\setminus{\bm0},
 \end{equation}
 for some constants \(\gamma>0\) and \(\tau>0\), and
\[
G_p(\bm y)
:=
\bigl(k(\overline H(p)+F(\bm y))\bigr)^{1/k}-p
\in H^r(\mathbb T^n)
\]
for some \(r>\tau+n/2\), 
then there exists a unique
(up to an additive constant) quasiperiodic corrector \(v\) satisfying
the cell problem
\[
\frac1k|p+v'(x)|^k-f(x)=\overline H(p),
\qquad x\in\mathbb R.
\]
\end{enumerate}
\end{theorem}

\begin{proof}
For (i), we extend the periodic result for \(k=2\)
in \cite[Example~4.1]{tran2021hamilton} to the quasiperiodic setting with
general \(k>1\). We next prove~(ii). Since \(\overline{H}\) is even, it suffices to consider \(p\geq p_0\). Set $\mu=\overline H(p)\geq0$
and define
\[
    X_p(t)
    =
    \int_0^t
    \bigl(k(\mu+F(\mathcal{C}(x)))\bigr)^{1/k}\,dx,
    \qquad t\in\mathbb{R}.
\]
Since
\(x\mapsto(\mu+F(\mathcal{C}(x)))^{1/k}\) is quasiperiodic,
the parameterized ergodic theorem
\cite[Proposition~4.2]{jiang2024numerical} gives
\begin{equation}\label{eqn:appendix_birkhoff}
    \lim_{t \to \infty} \frac{X_p(t)}{t} = \lim_{t \to \infty} \frac{1}{t} \int_0^t \bigl(k(\mu+F(\mathcal{C}(x)))\bigr)^{1/k} \, dx =  \int_{\mathbb{T}^n} \bigl( k (\mu + F(\bm{y})) \bigr)^{1/k}  d\bm{y} = p.
\end{equation}
Then we can construct a quasiperiodic function
\[
v'(x) = \bigl(k(\mu+F(\mathcal{C}(x)))\bigr)^{1/k} - p,
\]
whose antiderivative $v$ is also quasiperiodic under the Diophantine condition \eqref{eqn:dio_condition}.
Now compute
\[
\frac{1}{k} |p + v'(x)|^k = \frac{1}{k} \left| \bigl(k(\mu+F(\mathcal{C}(x)))\bigr)^{1/k} \right|^k = \mu + F(\mathcal{C}(x)) = \mu + f(x),
\]
since $f(x) = F(\mathcal{C}(x))$. Therefore,
\begin{equation} 
    \frac{1}{k} |p + v'(x)|^k - f(x) = \mu = \overline{H}(p),
\end{equation}
which proves that $v$ is an exact corrector for the cell problem. 

To prove uniqueness, let \(\widetilde v\) be another quasiperiodic corrector.
Since \(\mathcal M\{\widetilde v'\}=0\) and
\[
|p+\widetilde v'|
=
\bigl(k(\mu+ f)\bigr)^{1/k},
\]
the identity
\[
\mathcal M\{|p+\widetilde v'|\}
=
\mathcal M\{p+\widetilde v'\}
=
p
\]
implies \(p+\widetilde v'\ge0\), and hence
\(\widetilde v'=v'\). Therefore, \(\widetilde v-v\) is constant, which proves uniqueness up to an
additive constant.

\end{proof}
Based on \Cref{thm:keh}, the effective Hamiltonian \(\overline H\) is
determined by the corrector problem \eqref{eq:corrector_problem}, whose
direct solution is difficult because both the corrector and
\(\overline H\) are unknown. We therefore compute \(\overline H\) from
the long-time behavior of \eqref{eqn:long-time} using the SL--FPR
scheme. Following Qian and Rorro
\cite{rorro2006approximation,qian2003two}, the effective Hamiltonian is
then obtained by averaging \(w_t\), \textit{i.e.},

\[
    \overline v^{\,n}
    =
    \frac{h}{|\Omega|}\sum_i v_i^n,
    \qquad
    v_i^n
    =
    \frac{w_i^n-w_i^{n-1}}{\tau},
\]
where \(w_i^n\) is computed by \eqref{eqn:sls_fpr}.
For a prescribed tolerance \(\tilde{\varepsilon}>0\), the iteration terminates
when
\[
    |\overline v^{\,n+1}-\overline v^{\,n}|
    \leq\tilde{\varepsilon}.
\]
The converged value \(\overline v\) then gives
\begin{equation}
\label{eqn:hamiltonian_update}
    \overline H(p)=c_0-\overline v.
\end{equation}
Typically, we initialize the update process with $c_0 = -\min f$.



\section{Numerical tests}\label{sec:numerical_tests}
In this section, we present two classes of numerical tests to demonstrate the accuracy and adaptability of the SL--FPR scheme.  The first class concerns the numerical solution of QHJEs, while the second concerns the approximation of effective Hamiltonians.
 All experiments were conducted using MATLAB R2024a on a laptop equipped with an Intel Core 2.20GHz CPU and 16GB of RAM.

For the QHJE tests, the error at the final time \(t_M\) is measured by
\begin{equation*}
    e_u
    :=
    \max_{1\leq i\leq N}
    \left|
        u_i^M-u(x_i,t_M)
    \right|.
\end{equation*}
For the effective-Hamiltonian tests, let
\[
    \mathcal P_{\Delta p}
    :=
    \{p_j:\,p_j=j\Delta p,\ 0\leq p_j\leq4\}.
\]
Whenever a theoretical or high-accuracy reference
\(\overline H_{\rm exa}\) is available, let
\(\overline H_{\rm app}\) denote the effective Hamiltonian computed
using the SL--FPR scheme. We then define
\begin{equation*}
    e_p(p_j)
    :=
    \left|
        \overline H_{\rm app}(p_j)
        -
        \overline H_{\rm exa}(p_j)
    \right|,
    \qquad
    e_N
    :=
    \max_{p_j\in\mathcal P_{\Delta p}}
    e_p(p_j),
\end{equation*}
where \(N=2\pi/h\). The error order is calculated by
\begin{equation*}
    \kappa
    :=
    \frac{\log(e_1/e_2)}
         {\log(h_1/h_2)}.
\end{equation*}
\subsection{Quasiperiodic Hamilton--Jacobi equations}

In the first test, we assess the accuracy of the proposed SL--FPR scheme. 
Specifically, we consider the QHJEs 
\eqref{eqn:QHJE} posed on the domain $\Omega = [-99\pi,\,101\pi]$ with the quasiperiodic Hamiltonian
\[
H(x,p)=\frac{1}{2}|p|^2 - f(x),
\]
where
$f(x)=\displaystyle\frac{1}{2}\bigl(\sin x + \sin(\sqrt{2}\,x)\bigr)^2 + 1$.
The exact solution is explicitly given by
\[
u(x,t)
=
-\cos x
-\frac{1}{\sqrt{2}}\cos(\sqrt{2}\,x)
+ t.
\]
The primary objective of this experiment is to verify the spatial
convergence order of the SL--FPR scheme. Here, the time step is
chosen sufficiently small, namely $\tau = 10^{-7}$, such that the temporal
discretization error is negligible and the overall error is dominated by
the FPR interpolation in space. The spatial errors and corresponding convergence orders for the SL–FPR(1), SL–FPR(2), and SL–FPR(3) schemes are presented in \Cref{tab:QHJE-conv}. The numerical results indicate convergence rates of approximately 2, 3, and 4, respectively, confirming that these schemes achieve the expected second-, third-, and fourth-order spatial accuracy, respectively. Moreover, increasing the FPR interpolation order yields smaller errors
on the same mesh and a higher rate of error decay under mesh refinement.  These results are consistent with the convergence analysis and the reliability of the proposed SL–FPR schemes for solving QHJEs.

\begin{table}[!htbp]
\centering
\footnotesize
\caption{Error $e_u$ and convergence order $\kappa$ of SL--FPR(1), (2), (3) for solving \eqref{eqn:QHJE} with $\tau = 10^{-7}$.}\label{tab:QHJE-conv}
\begin{tabular}{|c|l|c|c|c|c|}
\hline
 & $h$ & $2\pi/20$ & $2\pi/40$ & $2\pi/80$ & $2\pi/160$ \\ 
\hline
\multirow{2}{*}{SL--FPR(1)} 
& $e_u$       & 2.92E-01 & 1.05E-01 & 2.65E-02 & 6.60E-03 \\
& $\kappa$  & --       & 1.47     & 1.98     & 2.01 \\
\hline
\multirow{2}{*}{SL--FPR(2)} 
& $e_u$       & 7.34E-02 & 1.42E-02 & 1.78E-03 & 2.22E-04 \\
& $\kappa$  & --       & 2.37     & 2.99     & 3.00 \\
\hline
\multirow{2}{*}{SL--FPR(3)} 
& $e_u$       & 5.87E-02 & 6.01E-03 & 3.76E-04 & 2.35E-05 \\
& $\kappa$  & --       & 3.29     & 3.99     & 4.00 \\
\hline
\end{tabular}
\end{table}

\subsection{Quasiperiodic effective Hamiltonians}
We next investigate the numerical approximation of the effective
Hamiltonian associated with
\begin{equation*}
    H(x, p)
    =
    \frac{1}{k}|p|^k-f(x),
    \qquad k>1.
\end{equation*}
The experiments serve three purposes. We first consider the quadratic case
\(k=2\), for which an established homogenization formula is available, to
validate the computation of \(\overline H\) by the SL--FPR scheme. We then
consider \(k\neq2\) to examine the extension established in
\Cref{thm:keh}. Finally, we consider low-regularity quasiperiodic potentials
to assess the robustness of the method beyond smooth  cases.

\subsubsection{Case 1: $k = 2$}\label{sec:k2_case}
~ We first consider \(k=2\), for which the homogenization result in
\cite[Lemma~3.1]{hu2024polynomial} provides a reference solution.
This allows us to verify the reliability of the SL--FPR scheme in computing effective Hamiltonians.  We consider the corresponding  corrector problem
\begin{equation}\label{eqn:corrector_k2}
    H(x, p + v') = \frac{1}{2}|p + v'|^2 - f(x) = \overline{H}(p), \quad x\in\mathbb{R} 
\end{equation}
with quasiperiodic potential term
$
f (x) = \cos (2\pi x) + \cos (2\pi\sqrt{2} x),\, x\in\mathbb{R}.
$
From the definition of quasiperiodic functions, the parent function is 
$
F(\bm{y}) = \cos(y_1) + \cos(y_2),\, \bm{y} = (y_1, y_2)\in\mathbb{T}^2. 
$
Thus, we can calculate the exact Hamiltonian by the integration formula
\begin{equation}\label{eqn:reference_value}
    \overline{H}(p) = \left\{\begin{aligned}
        &-\min_{x\in\mathbb{R}} f, \quad|p|\leq p_0\\
        &\lambda:  p=\int_{\mathbb{T}^2}\sqrt{2(\lambda+F(\bm{y}))}\,d\bm{y}, \quad |p| > p_0
    \end{aligned}
    \right.
\end{equation}
with $p_0 = \displaystyle\int_{\mathbb{T}^2}\sqrt{2(F(\bm{y})-\min_{x\in\mathbb{R}} f)}\,d\bm{y}.$
The reference value \(\overline H_{\rm exa}\) is computed from
\eqref{eqn:reference_value} by Gauss--Legendre quadrature with tolerance \(10^{-10}\).
For the SL--FPR approximation, we take \(\tau=h/2\) and
\(\tilde{\varepsilon} = h\tau\).

We first compare \(\overline H_{\rm app}\) with
\(\overline H_{\rm exa}\) in \Cref{fig:k2_hamiltonian_value}.
The two curves agree closely over \(p\in[0,4]\) and reproduce the two
regimes predicted by \eqref{eqn:reference_value}. The effective Hamiltonian
remains nearly constant at
\(\overline H=-\min\limits_{x\in\mathbb R}f=2\) for \(0\le p<p_0\),
and exhibits quadratic growth for \(p>p_0\).
The corresponding pointwise error \(e_p\) is shown in
\Cref{fig:k2_hamiltonian_error}. The error remains small over most of the
sampled interval but becomes significantly larger near the transition
point \(p_0\). To assess the convergence of SL--FPR scheme, we examine the maximum error \(e_N\).
As shown in \Cref{tab:error_table}, \(e_N\) decreases monotonically under
mesh refinement, with observed first-order convergence. Since the
effective Hamiltonian is extracted from the long-time solution of the
QHJE computed by the SL--FPR scheme, its numerical accuracy is closely
related to the accuracy of the SL--FPR approximation. For SL--FPR(1), \Cref{thm:convergence} gives
\(\mathcal O(h^2)+\mathcal O(\tau)\). Since \(\tau=h/2\), the resulting
first-order accuracy is consistent with the observed convergence of
\(\overline H_{\rm app}\).

\begin{figure}[!htbp]
\centering

\subfigure[]{
    \includegraphics[width=5cm]{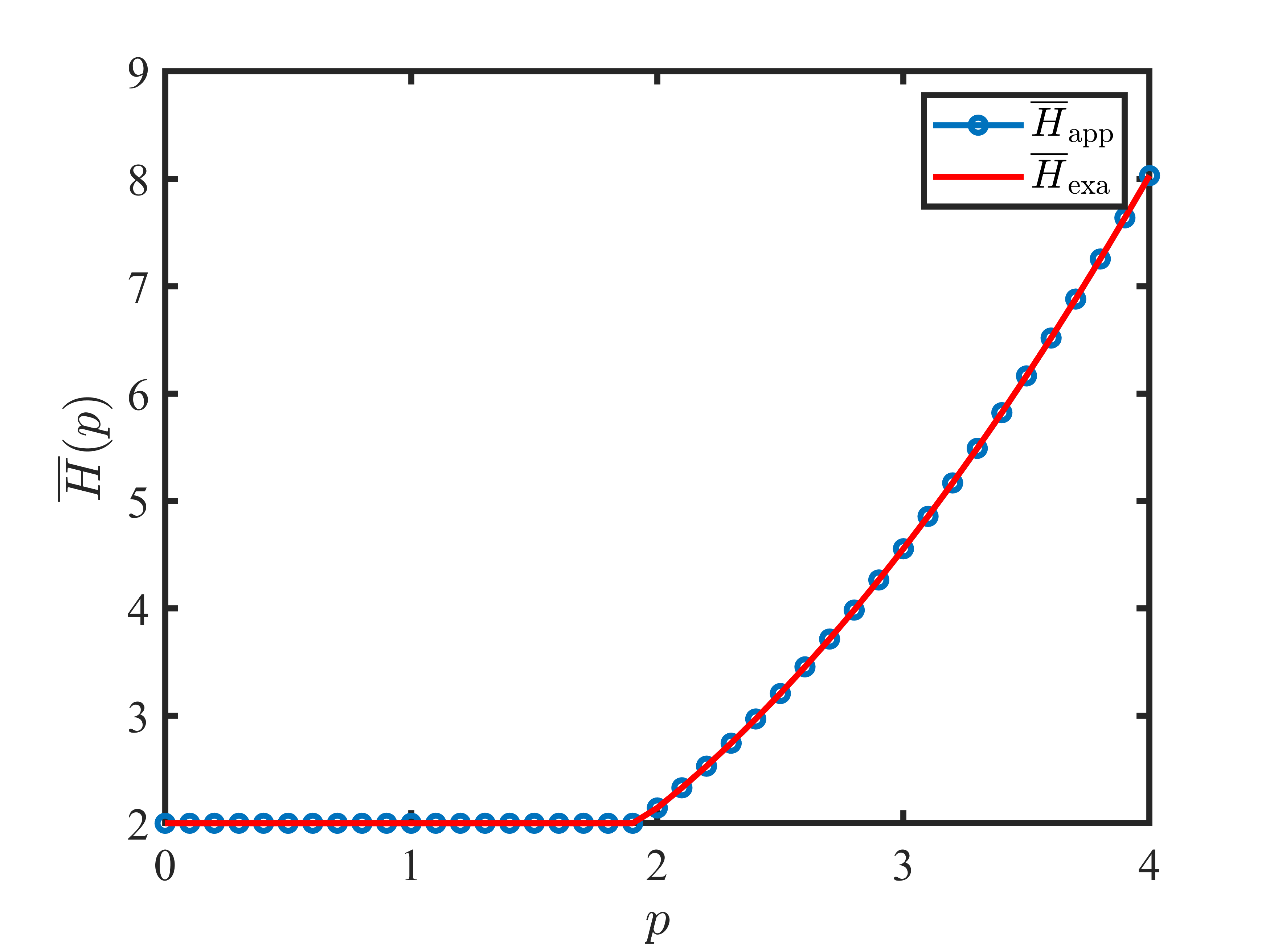}
    \label{fig:k2_hamiltonian_value}
}
\hspace{5mm}
\subfigure[]{
    \includegraphics[width=5cm]{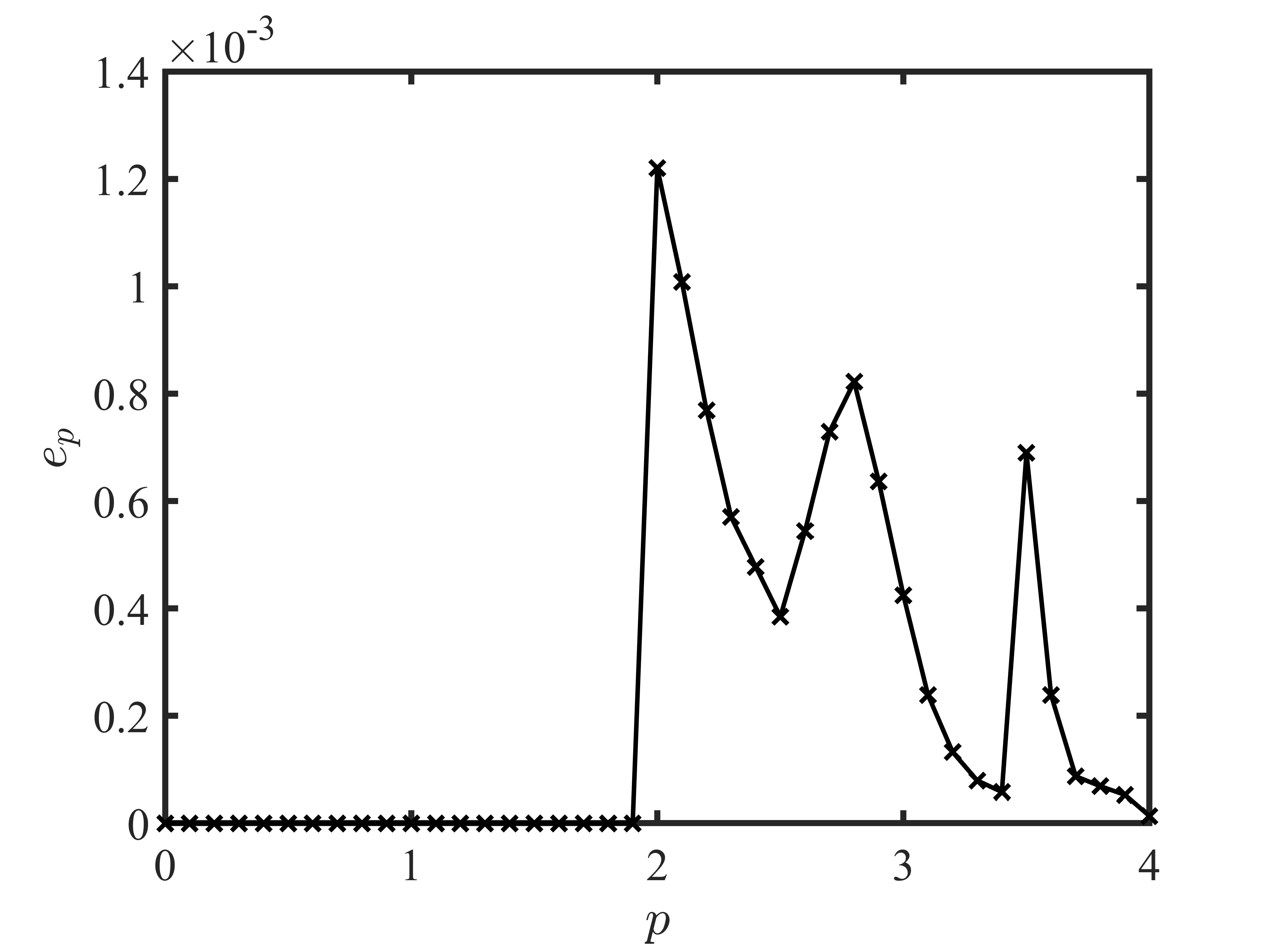}
    \label{fig:k2_hamiltonian_error}
}

\caption{
Numerical approximation and pointwise error of the effective Hamiltonian
for the smooth quasiperiodic potential with \(k=2\).
\textup{(a)} Comparison between 
\(\overline H_{\rm app}\) and 
\(\overline H_{\rm exa}\) for \(p\in[0,4]\) with \(h=2\pi/160\);
\textup{(b)} corresponding pointwise error
\(e_p\).
}
\label{fig:quasiperiodic_effective_hamiltonian}

\end{figure}

\begin{table}[!htbp]
\centering
\footnotesize
\caption{Errors \(e_N\) and convergence orders \(\kappa\) for \(\Delta p=0.1\).}
\label{tab:error_table}
\begin{tabular}{|l|c|c|c|c|c|}
\hline
\(N\) & \(10\) & \(20\) & \(40\) & \(80\) & \(160\) \\
\hline
\(e_N\)
& 1.96E-02
& 9.82E-03
& 4.91E-03
& 2.45E-03
& 1.22E-03 \\
\hline
\(\kappa\)
& --
& 0.99
& 1.00
& 1.00
& 1.00 \\
\hline
\end{tabular}
\end{table}




\subsubsection{Case 2: $k>1, \, k\neq 2$}
~The numerical results in \Cref{sec:k2_case} validate the SL--FPR
scheme for the case \(k=2\), where the existing homogenization result
provides a theoretical reference. Having
confirmed its reliability in this case, we apply the scheme to
\(k>1\), \(k\neq2\), to numerically verify the extended homogenization
result in \Cref{thm:keh}. The corresponding  corrector equation reads 
\begin{equation}\label{eqn:cubic}
    H(x, p+v') = \frac{1}{k}|p+v'|^k - f(x) = \overline{H}(p), \quad k>1.
\end{equation}
Here, we select $f(x) = 2-\cos (2\pi x) - \cos(2\pi\sqrt{2}x)$. We first apply the SL--FPR scheme to solve \eqref{eqn:cubic} with $k = 3$.




Since no independently established reference solution is available for
\(k\neq2\), we assess the computed effective Hamiltonian through
theoretical consistency and numerical convergence. For the potential
satisfying \(\min\limits_{x\in\mathbb R}f(x)=0\), the effective Hamiltonian
satisfies
\(\min\limits_{p\in\mathbb R}\overline H(p)=\overline H(0)=0\)
and exhibits a flat region near its minimum. Moreover, since \(H\) is
convex in \(p\), \(\overline H\) is also convex
\cite[Theorem~5.20]{tran2021hamilton}. As shown in
\Cref{fig:quasi_cube}, the numerical approximation
\(\overline H_{\rm app}\) reproduces both properties.
The numerical convergence is further examined in
\Cref{fig:cubeerr1}, where the solution with \(h=2\pi/160\) is taken as
the reference approximation. As the mesh size is halved, the errors
decrease approximately by a factor of two, indicating first-order
convergence.

\begin{figure}[!htbp]
    \centering
    \subfigure[]{
        \includegraphics[width=5cm]{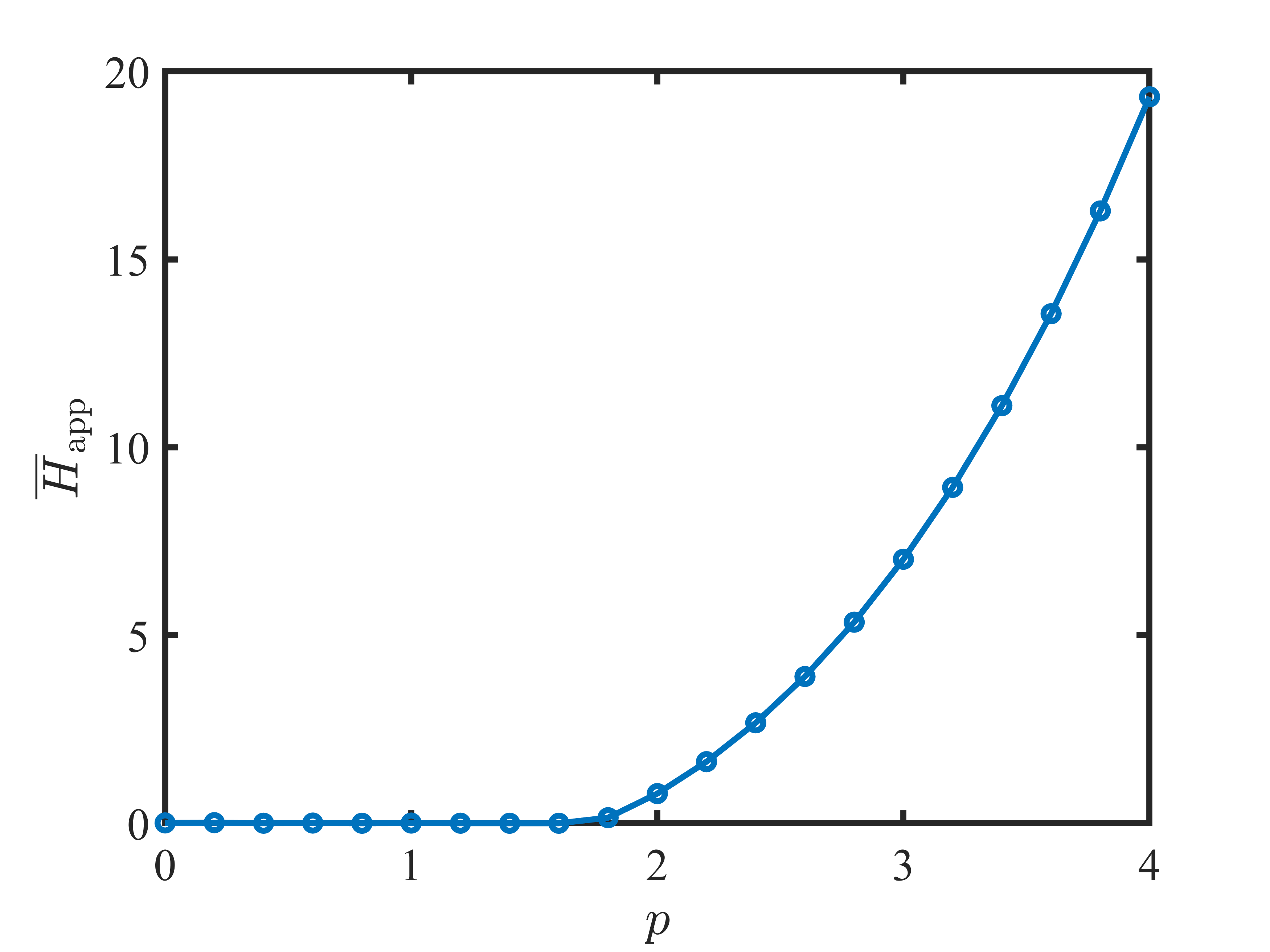}
        \label{fig:quasi_cube}
    }
    \hspace{2mm}
    \subfigure[]{
        \includegraphics[width=5cm]{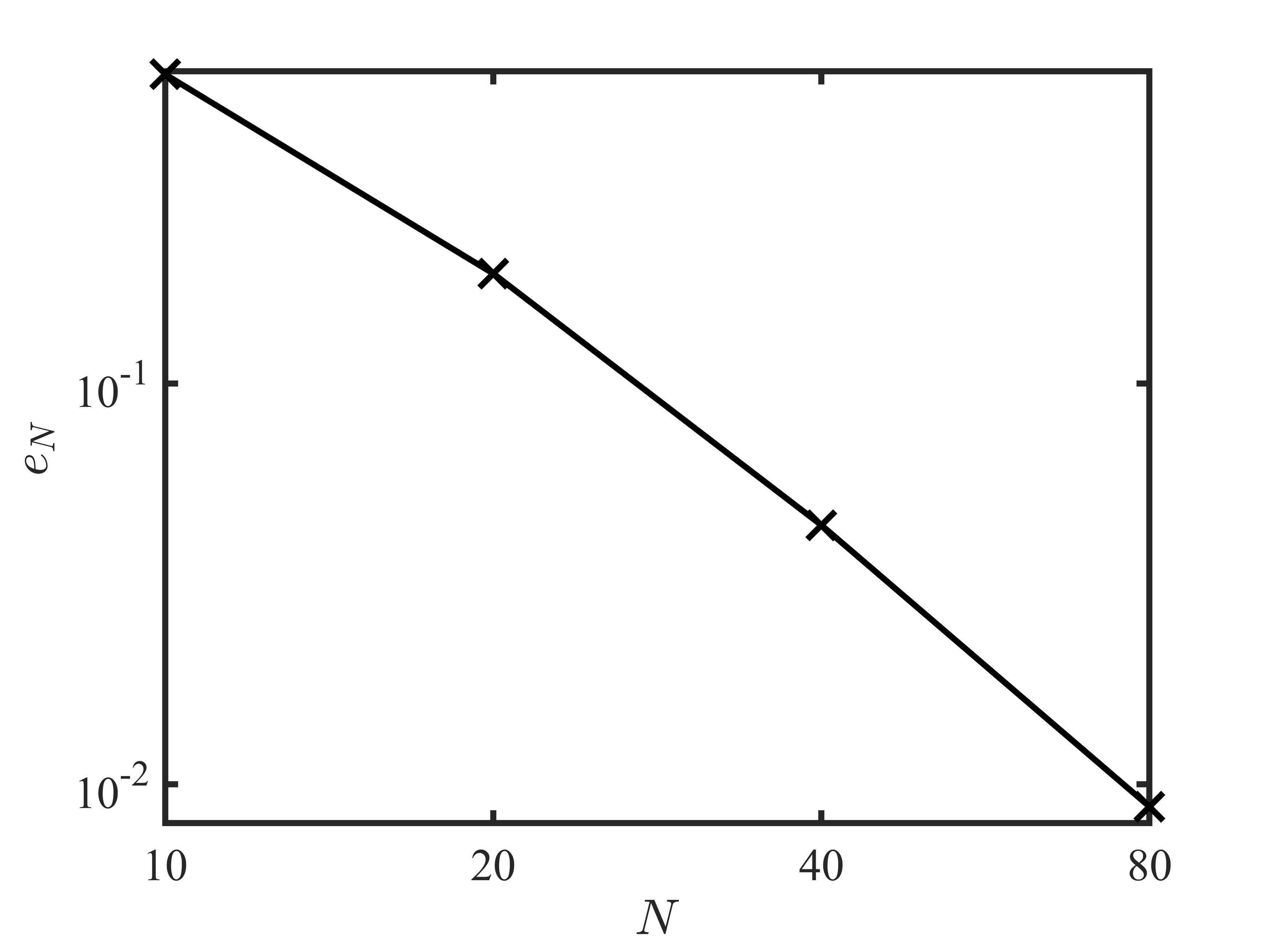}
        \label{fig:cubeerr1}
    }
   \caption{Numerical results for $\overline{H}$ for the smooth quasiperiodic potential in the case $k=3$. \textup{(a)} Computed effective Hamiltonian $\overline H_{\rm app}$ for $p\in[0,4]$ with $h=2\pi/160$ and $\Delta p=0.2$; \textup{(b)} errors $e_N$ under mesh refinement for \(N=10,20,40,80\), with the
finest-grid solution at \(N=160\) taken as the reference approximation.}
    \label{fig:quasiperiodic_cube_Hamiltonian}
\end{figure}


To further validate our theoretical analysis, we extend the numerical experiments to a broader range of exponents with $k = 1.5, 2.5, 3$ and $4$. As shown in \Cref{fig:kqhe}, the computed effective Hamiltonians
consistently exhibit the flat region and convexity established in \Cref{thm:keh}, followed by the expected \(k\)-th power growth as \(p\) increases. The agreement for both integer and fractional values
of \(k\) provides further numerical evidence for the homogenization
result for general \(k>1\).

\begin{figure}[htbp]
    \centering
    {
        \includegraphics[width=5.5cm]{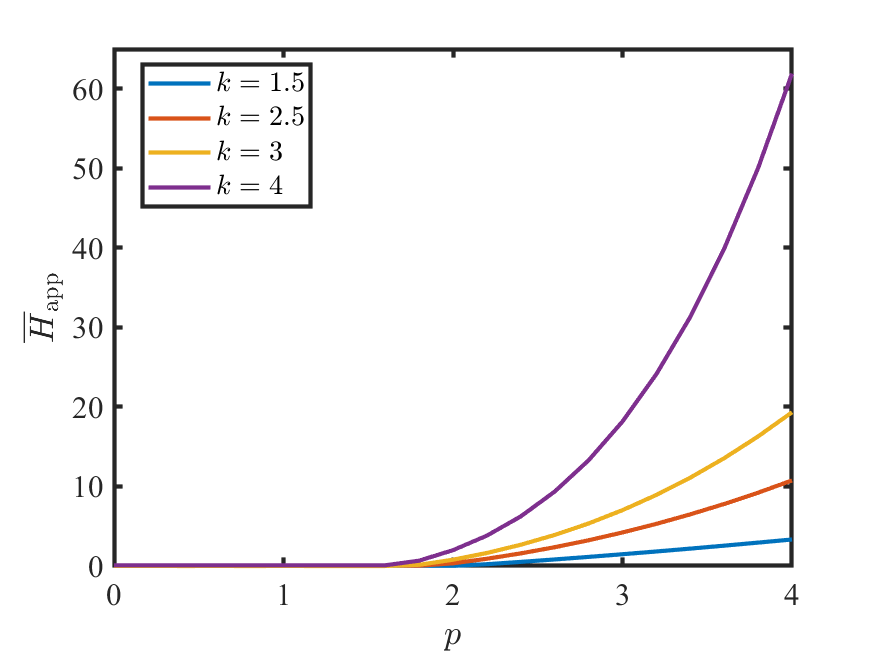}
    }
   
    \caption{
        Numerical effective Hamiltonians for $p \in [0, 4]$ with $\Delta p = 0.2$ under different $k$.
    }
    \label{fig:kqhe}
\end{figure}

\subsubsection{Case 3: Low-regularity quasiperiodic potentials}
~We finally consider two nonsmooth quasiperiodic potentials to assess the
performance of the SL--FPR scheme in low-regularity settings.

\textbf{Example~1: $C^0$ quasiperiodic potential test.}
~We first consider the continuous but nonsmooth potential
\begin{equation}\label{eq:c0_potential}
f(x)=2-\lvert\cos(2\pi x)\rvert-\lvert\cos(2\sqrt{2}\pi x)\rvert.
\end{equation}
The function \(f\) belongs to \(C^0(\mathbb{R})\) and is Lipschitz continuous, but fails to be differentiable at some points.
The exact effective Hamiltonian $\overline{H}(p)$ 
can still be computed by the  formula \eqref{eqn:reference_value}.

As shown in \Cref{fig:quasiperiodic_c0_Hamiltonian}, the SL--FPR scheme remains accurate
for this Lipschitz-continuous but nonsmooth potential. The computed
effective Hamiltonian is in good agreement with the reference value,
and the errors retain the expected first-order decay. This example demonstrates the robustness of the SL--FPR scheme for
the numerical homogenization of low-regularity QHJEs.

\begin{figure}[!htbp]
    \centering
    \subfigure[]{
        \includegraphics[width=5cm]{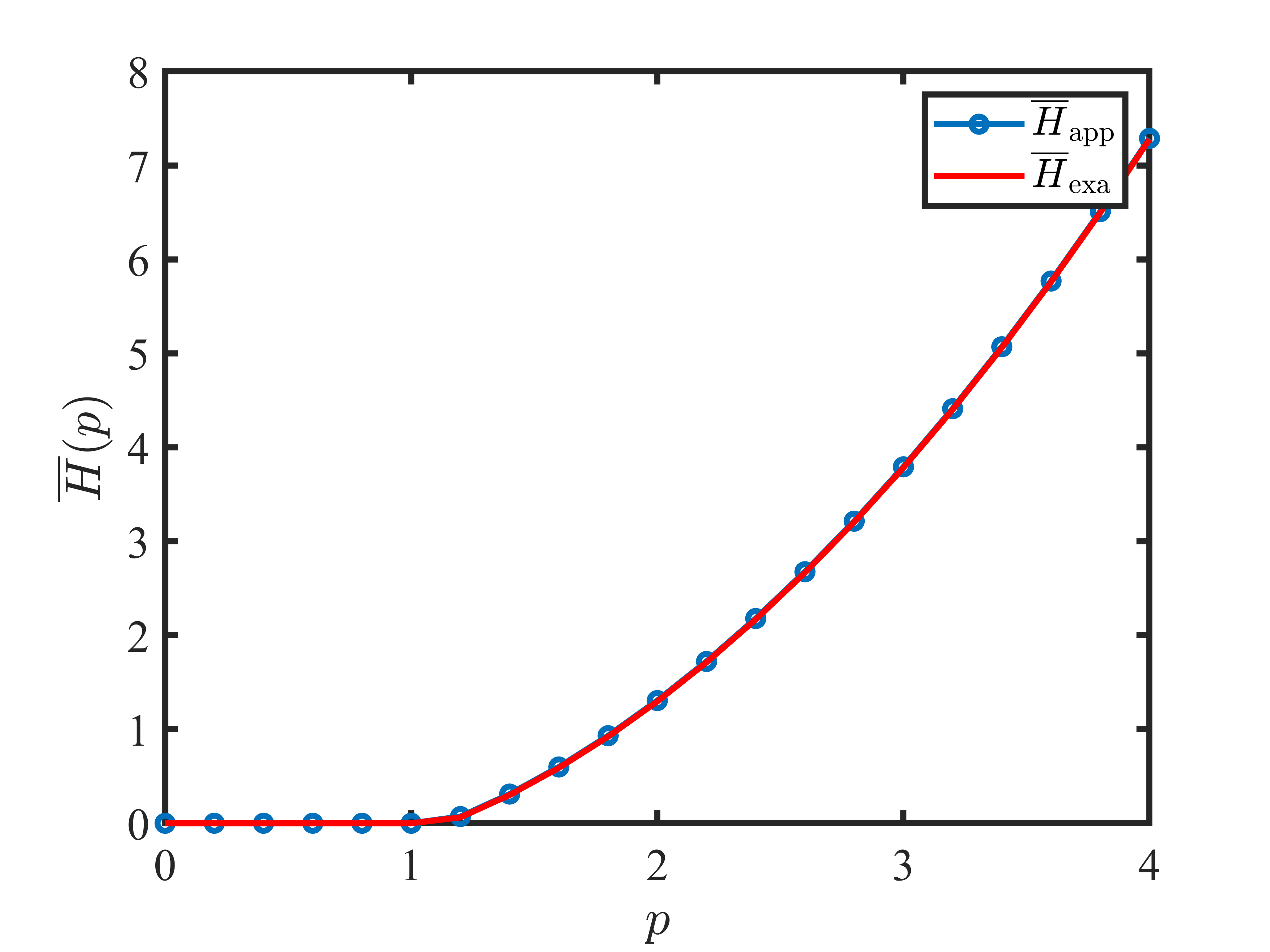}
        \label{fig:quasi_c0}
    }
    \hspace{2mm}
    \subfigure[]{
        \includegraphics[width=5cm]{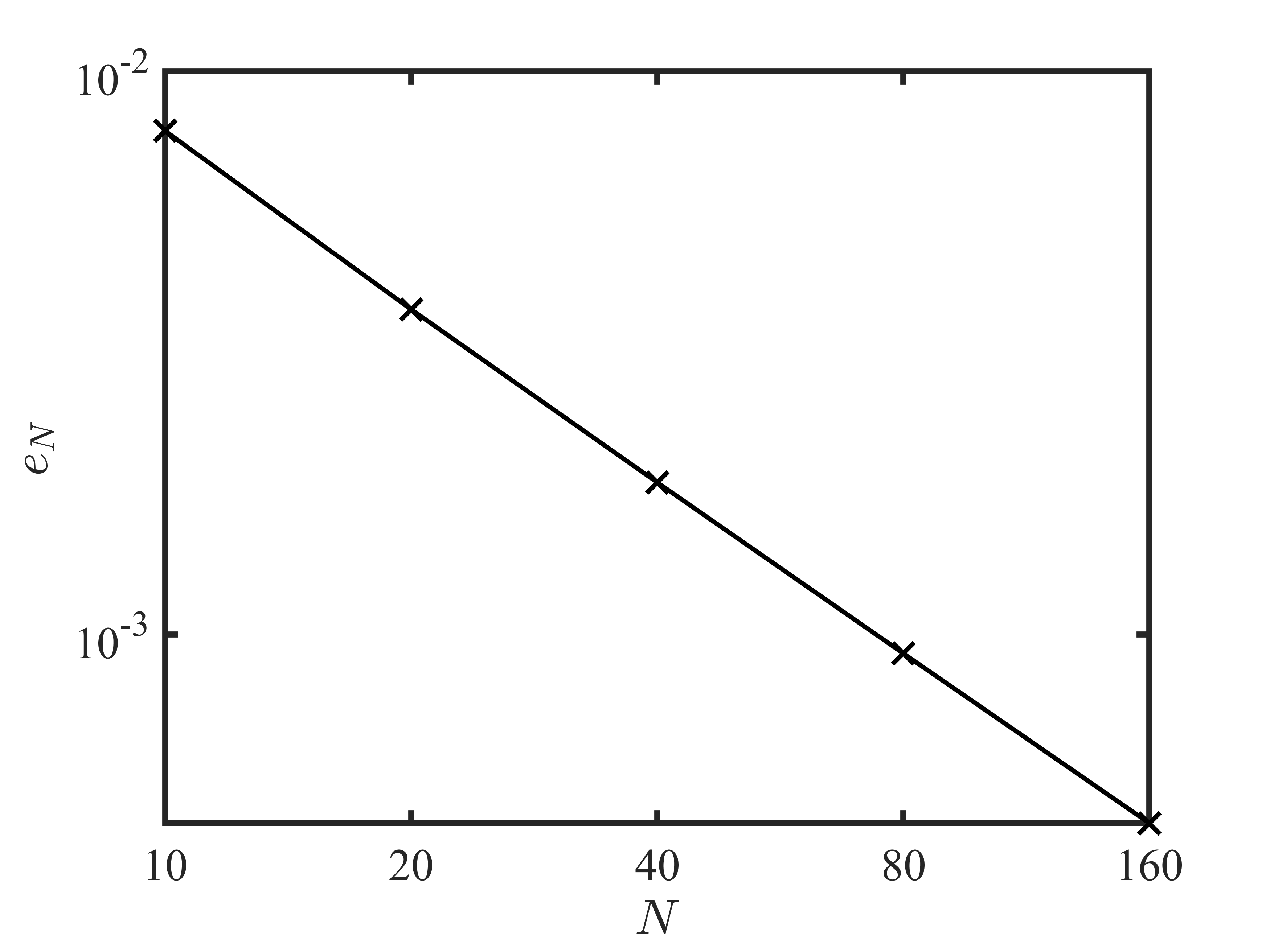}
        \label{fig:c0err}
    }
    \caption{Numerical effective Hamiltonian for a $C^0$ quasiperiodic potential \eqref{eq:c0_potential}. \textup{(a)} Comparison between \(\overline H_{\rm exa}\) and \(\overline H_{\rm app}\) for \(p\in[0,4]\) with
\(h=2\pi/160\) and \(\Delta p=0.2\); \textup{(b)} errors $e_N$ between the reference solution and the computed approximation for $N=10,20,40,80,160$.}
    \label{fig:quasiperiodic_c0_Hamiltonian}
\end{figure}

\textbf{Example~2: Fibonacci quasiperiodic potential test.}
~We then consider a discontinuous, piecewise-constant Fibonacci-type
quasiperiodic potential. For numerical interpolation, each jump is
replaced by a short linear transition, yielding the Lipschitz-continuous
but nonsmooth piecewise-linear approximation shown in
\Cref{fig:fibonacci_regularized}. The sharp transitions inherited from the original discontinuities
provide a more challenging low-regularity test than Example~1.

The computed effective Hamiltonian is shown in
\Cref{fig:fibonacci_hamiltonian}, while the corresponding numerical
errors $e_N$ are presented in
\Cref{fig:fibonacci_error}.
The computed effective Hamiltonian exhibits stable convergence under
mesh refinement. Although order reduction is observed on coarse meshes, the expected
convergence rate is recovered as the mesh is refined. This provides
further evidence of the applicability of the SL--FPR scheme in the
low-regularity case.
\begin{figure}[!htbp]
\centering

\subfigure[]{
    \includegraphics[width=0.5\textwidth]
    {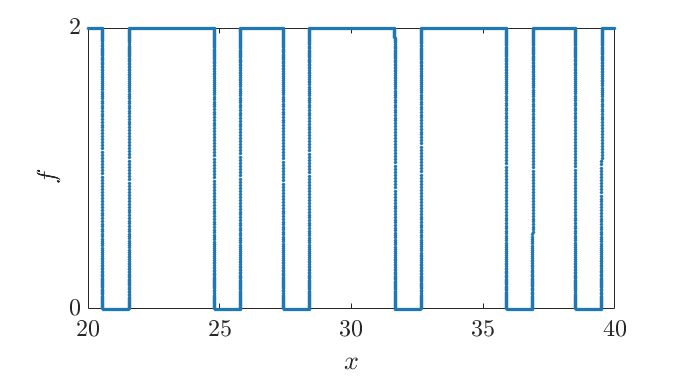}
    \label{fig:fibonacci_regularized}
}

\vspace{-2mm}

\subfigure[]{
    \includegraphics[width=0.35\textwidth]
    {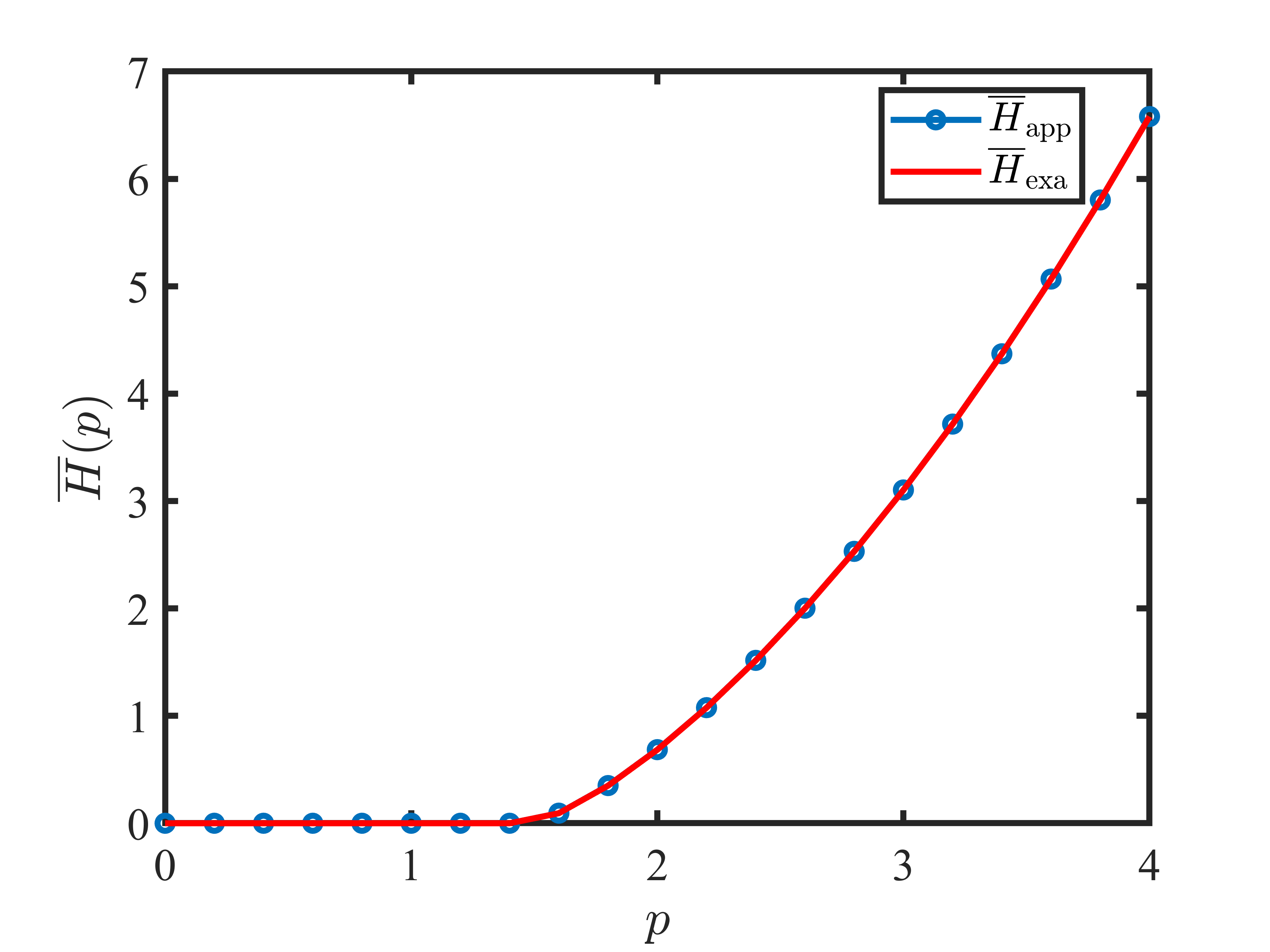}
    \label{fig:fibonacci_hamiltonian}
}
\hspace{4mm}
\subfigure[]{
    \includegraphics[width=0.35\textwidth]
    {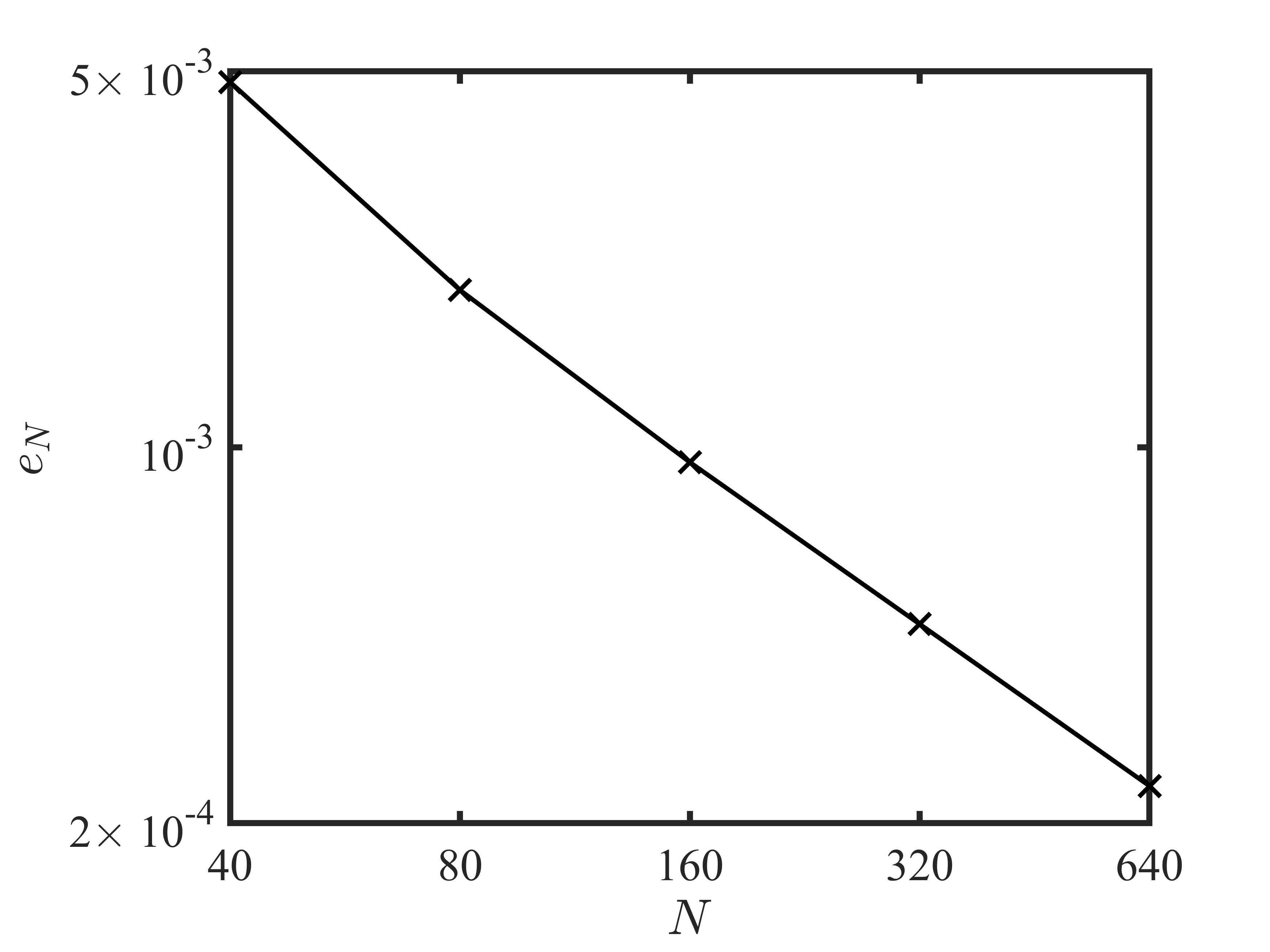}
    \label{fig:fibonacci_error}
}
\caption{
Numerical results for the Fibonacci-type quasiperiodic potential.
\textup{(a)} Regularized potential on \(x\in[20,40]\);
\textup{(b)} comparison of \(\overline H_{\rm exa}\) and
\(\overline H_{\rm app}\) on \(p\in[0,4]\) with
\(h=2\pi/160\) and \(\Delta p=0.2\);
\textup{(c)} error \(e_N\) under mesh refinement.}
\label{fig:quasiperiodic_fibonaci_Hamiltonian}

\end{figure}


\section{Conclusions and outlooks}\label{sec:conclusion}
We developed a computable finite-size formulation and a
high-accuracy numerical scheme for QHJEs with \(k\)-power convex Hamiltonians. By reformulating the problem with quasiperiodic boundary conditions and combining SL discretization with the FPR method, we established an error estimate of the proposed scheme and demonstrated its effectiveness in computing effective Hamiltonians. Numerical tests confirmed the accuracy of the method, including applications to Hamiltonians without known analytical solutions before this work. 

 Future work includes extending the homogenization results and the SL--FPR scheme from one dimension to higher dimensions. Moreover, developing results for the non-convex setting remains a challenging problem. Finally, extending the proposed approach to other classes of multiscale
problems, such as quasiperiodic Schr\"odinger and Navier--Stokes equations,
requires further investigation.

\end{document}

%% file: ex_shared.tex
\headers{Method for quasiperiodic HJE \& homogenization}
{K. Jiang, M. Li, J. Zhang, L. Zhang}

\title{ A High-Accuracy Numerical Homogenization Method for Quasiperiodic Hamilton--Jacobi Equations
\thanks{Submitted to xxx.
}
}

\author{Kai Jiang\thanks{
Wuhan Institute for Math \& AI, Wuhan University, Wuhan, Hubei, 430072, China
  (\email{kaijiang@whu.edu.cn}).}
\and Meng Li \thanks{Hunan Key Laboratory for Computation and Simulation in Science and Engineering, School  of Mathematics  and  Computational  Science,  Xiangtan  University,  Xiangtan,  Hunan,  411105, China
(\email{limeng@smail.xtu.edu.cn, zhangjuan@xtu.edu.cn}).
} 
\and Juan Zhang\footnotemark[2]
\and Lei Zhang\thanks{School of Mathematical Sciences, Institute of Natural Sciences, MOE-LSC, Shanghai Jiao Tong University, Shanghai, 200240, China
  (\email{lzhang2012@sjtu.edu.cn}).}}

\usepackage{mathrsfs,amsmath,amssymb,bm}
\usepackage{lipsum}
\usepackage{amsfonts}
\usepackage{graphicx, subfigure}
\usepackage{epstopdf}
\usepackage{makecell, rotating, bbding}
\usepackage{multirow}
\usepackage{algorithmic, algorithm}
\usepackage{enumerate}
\usepackage{booktabs}
\usepackage[misc]{ifsym}
\usepackage{mathtools}
\usepackage[draft]{changes}
\usepackage{lipsum}
\usepackage{amsopn}

\usepackage{aligned-overset}
\usepackage{mathtools}
\usepackage{nicematrix}
\usepackage{tikz}
\usepackage{dashbox}
\usetikzlibrary{calc}
\usepackage{hyperref}
\usepackage{enumitem}
\usepackage{enumerate}

\newcommand{\ave}[1]{\langle #1 \rangle} 
\newtheorem{thm}{Theorem}[section]

\newtheorem{remark}[thm]{Remark}

\newsiamremark{hypothesis}{Hypothesis}
\crefname{hypothesis}{Hypothesis}{Hypotheses}
\newsiamthm{claim}{Claim}

\makeatletter
\newcommand*{\addFileDependency}[1]{
  \typeout{(#1)}
  \@addtofilelist{#1}
  \IfFileExists{#1}{}{\typeout{No file #1.}}
}
\makeatother


%% file: main.bbl
\begin{thebibliography}{99}

\bibitem{bardi1997optimal}
{M. Bardi and I. Capuzzo-Dolcetta},
{\it Optimal Control and Viscosity Solutions of Hamilton--Jacobi--Bellman Equations},
Birkh\"auser, Boston, 1997.

\bibitem{concordel1996periodic}
{M. C. Concordel}, {\it Periodic homogenization of Hamilton--Jacobi equations: additive eigenvalues and variational formula}, Indiana Univ. Math. J., {\bf45}(1996), 1095--1117.

\bibitem{concordel1997periodic}
{M. C. Concordel}, {\it Periodic homogenisation of Hamilton--Jacobi equations. II. Eikonal equations}, Proc. Roy. Soc. Edinburgh Sect. A, {\bf127}(1997), 665--689.


\bibitem{contreras1998lagrangian}
{G. Contreras, R. Iturriaga, G. P. Paternain and M. Paternain},
{\it Lagrangian graphs, minimizing measures and Ma\~n\'e's critical values},
Geom. Funct. Anal., {\bf8}(1998), 788--809.

\bibitem{qian2018min}
{J. Qian, H. V. Tran and Y. Yu}, {\it Min-max formulas and other properties of certain classes of nonconvex effective Hamiltonians}, Math. Ann., {\bf372}(2018), 91--123.

\bibitem{falcone2008on}
{M. Falcone and M. Rorro}, {\it On a variational approximation of the effective Hamiltonian}, in {\it Numerical Mathematics and Advanced Applications}, Springer, (2018), 719--726.

\bibitem{gomes2004computing}
{D. A. Gomes and A. M. Oberman}, {\it Computing the effective Hamiltonian using a variational approach}, SIAM J. Control Optim., {\bf43}(2004), 792--812.

\bibitem{deville1993smooth}
{R. Deville, G. Godefroy and V. Zizler}, {\it A smooth variational principle with applications to Hamilton--Jacobi equations in infinite dimensions}, J. Funct. Anal., {\bf111}(1993), 197--212.

\bibitem{ishii2000almost}
{H. Ishii}, {\it Almost periodic homogenization of Hamilton--Jacobi equations}, in {\it International Conference on Differential Equations}, {\bf1}(2000), 600--605.

\bibitem{tran2021hamilton}
{H. V. Tran}, {\it Hamilton--Jacobi equations: theory and applications}, Amer. Math. Soc., {\bf213}(2021).

\bibitem{hu2024polynomial}
{B. Hu, S. N. T. Son and J. Zhang}, {\it Polynomial convergence rate for quasiperiodic homogenization of Hamilton--Jacobi equations}, Comm. Partial Differential Equations,  {\bf50}(2024), 1--34.

\bibitem{giga2021existence}
{Y. Giga, H. Mitake, T. Ohtsuka and H. V. Tran}, {\it Existence of asymptotic speed of solutions to birth-and-spread type nonlinear partial differential equations}, Indiana Univ. Math. J., {\bf70}(2021), 121--156.

\bibitem{falcone2013semi}
{M. Falcone and R. Ferretti}, {\it Semi-Lagrangian approximation schemes for linear and Hamilton--Jacobi equations}, SIAM, 2013.

\bibitem{rorro2006approximation}
{M. Rorro}, {\it An approximation scheme for the effective Hamiltonian and applications}, Appl. Numer. Math., {\bf56}(2006), 1238--1254.

\bibitem{lions1982generalized}
{P. L. Lions}, {\it Generalized solutions of Hamilton--Jacobi equations}, Res. Notes Math., {\bf69}(1982).

\bibitem{lions1987homogenization}
{P. L. Lions, G. Papanicolaou and S. R. Srinivasa Varadhan}, {\it Homogenization of Hamilton--Jacobi equations}, unpublished work, 1987.

\bibitem{lions2003correctors}
{P. L. Lions and P. E. Souganidis}, {\it Correctors for the homogenization of Hamilton--Jacobi equations in the stationary ergodic setting}, Comm. Pure Appl. Math., {\bf56}(2003), 1501--1524.

\bibitem{evans2022partial}
{L. C. Evans}, {\it Partial differential equations}, Grad. Stud. Math., {\bf19}, Amer. Math. Soc., 2022.

\bibitem{crandall1983viscosity}
{M. Crandall and P. L. Lions}, {\it Viscosity solutions of Hamilton--Jacobi equations}, Trans. Amer. Math. Soc., {\bf277}(1983), 1--45.

\bibitem{osher1988fronts}
{S. Osher and J. A. Sethian}, {\it Fronts propagating with curvature-dependent speed: Algorithms based on Hamilton--Jacobi formulations}, J. Comput. Phys., {\bf79}(1988), 12--49.


\bibitem{li2003numerical}
{X. Li, W. Yan and C. K. Chan}, {\it Numerical schemes for Hamilton--Jacobi equations on unstructured meshes}, Numer. Math., {\bf94}(2003), 315--331.

\bibitem{jiang2023approximation}
{K. Jiang, S. Li and P. Zhang}, {\it On the approximation of quasiperiodic functions with Diophantine frequencies by periodic functions}, SIAM J. Math. Anal., {\bf57}(2025), 951--978.

\bibitem{jiang2025projection}
{K. Jiang, M. Li, J. Zhang and L. Zhang}, {\it Projection method for quasiperiodic elliptic equations and application to quasiperiodic homogenization}, SIAM J. Numer. Anal., {\bf63}(2025), 1962--1985.


\bibitem{jiang2014numerical}
{K. Jiang and P. Zhang}, {\it Numerical methods for quasicrystals}, J. Comput. Phys., {\bf256}(2014), 428--440.


\bibitem{jiang2024numerical}
{K. Jiang, S. Li and P. Zhang}, {\it Numerical methods and analysis of computing quasiperiodic systems}, SIAM J. Numer. Anal., {\bf62}(2024), 353--375.

\bibitem{jiang2024accurately}
{K. Jiang, Q. Zhou and P. Zhang}, {\it Accurately recover global quasiperiodic systems by finite points}, SIAM J. Numer. Anal., {\bf62}(2024), 1713--1735.

\bibitem{fan2025representation}
A. Fan, K. Jiang, and P. Zhang, 
{\it Representation of quasi-periodic functions and Hausdorff--Young inequalities for Besicovitch almost periodic functions}, 
arXiv preprint arXiv:2512.06821, 2025.

\bibitem{han2026accurately}
X. Han, K. Jiang, and M. Li, 
{\it Accurately computing quasiperiodic parabolic equations within finite-size domains via modeling quasiperiodic boundary conditions}, arXiv preprint arXiv:2608.22241, 2026.

\bibitem{corrias1995numerical}
{L. Corrias, M. Falcone and R. Natalini}, {\it Numerical schemes for conservation laws via Hamilton--Jacobi equations}, Math. Comp., 64 (1995), 555--580.

\bibitem{levine1984quasicrystals}
{D. Levine and P. J. Steinhardt}, {\it Quasicrystals: a new class of ordered structures}, Phys. Rev. Lett., {\bf53}(1984), 2477.




\bibitem{poincare1890problem}
{H. Poincar{\'e}}, {\it On the problem of three bodies and equations of dynamics}, Acta Math., {\bf13}(1890), A3--A270.


\bibitem{qian2003two}
{J. Qian}, {\it Two approximations for effective Hamiltonians arising from homogenization of Hamilton--Jacobi equations}, UCLA CAM Report, (2003), 03--39.


\bibitem{falcone2002semi}
{M. Falcone and R. Ferretti}, {\it Semi-Lagrangian schemes for Hamilton--Jacobi equations, discrete representation formulae and Godunov methods}, J. Comput. Phys., 175(2002), 559--575.

\bibitem{rockafellar1998variational}
{R. T. Rockafellar}, {\it Variational analysis}, Springer, Berlin, 1998.

\bibitem{sochi2014using}
{T. Sochi}, {\it Using the Euler--Lagrange variational principle to obtain flow relations for generalized Newtonian fluids}, Rheol. Acta, {\bf53}(2014), 15--22.

\bibitem{kamrin2014symmetry}
{K. Kamrin and J. D. Goddard}, {\it Symmetry relations in viscoplastic drag laws}, Proc. Roy. Soc. A, {\bf470}(2014), 20140434.

\bibitem{panov2017almost}
{E. Panov}, {\it On almost periodic viscosity solutions to Hamilton--Jacobi equations}, arXiv preprint arXiv:1707.00145, 2017.

\bibitem{einsiedler2017functional}
{M. Einsiedler and T. Ward}, {\it Functional analysis, spectral theory, and applications}, Springer, {\bf276}(2017).



\end{thebibliography}
